\documentclass[preprint,11pt]{elsarticle}
\usepackage[utf8]{inputenc}
\usepackage[T1]{fontenc}
\usepackage{amsfonts,amsmath,amssymb,amsthm,bbm,enumerate,enumitem,geometry,wrapfig,url,yfonts,fancyhdr}
\usepackage{multirow}
\usepackage{subfigure,graphicx,color,mathrsfs }
\everymath{\displaystyle}
\usepackage{ulem}

\newtheorem{tw}{Theorem}[section]	

\newtheorem{proposition}[tw]{Proposition}

\theoremstyle{definition} 
\newtheorem{definition}[tw]{Definition}

\newtheorem{example}[tw]{Example}
\newtheorem{remark}[tw]{Remark}
\renewenvironment{proof}[1][\textbf{Proof}]{\noindent\textbf{#1.} }{\hfill\qed\vspace{2.5ex}}

\newcommand\mN{{\mathbb N}}

\newcommand\agg{\mathsf{Agg}}
\newcommand\wmu{\widehat{\mu}}
\newcommand\wnu{\widehat{\nu}}

\newcommand\sA{\mathsf{A}}
\newcommand\waA[2][D]{\widehat{\sA}(#2|#1)}
\newcommand\aA[2][D]{\sA(#2|#1)}
\newcommand\aAi[3][D]{\sA^{#3}(#2|#1)}

\newcommand\bA{{\mathscr {A}}}

\newcommand\bF{\mathbf{F}} 
\newcommand\bM{\mathbf{M}} 
\newcommand\bEM{\widehat{\mathbf{M}}} 

\newcommand\cD{{\mathcal D}}

\newcommand\cH{{\mathcal H}}

\newcommand\cP{{\mathcal P}}
\newcommand\cR{{\mathcal R}}

\newcommand{\bs}[1]{B_{(#1)}}
\newcommand{\cs}[1]{C_{(#1)}}
\newcommand{\f}[1]{f_{(#1)}}
\newcommand{\fs}[1]{\aAi[#1]{f}{\inf}}

\newcommand\fC{{\mathfrak C}}
\newcommand\fR{{\mathfrak R}}

\newcommand{\rC}{\mathrm{C}}

\newcommand{\rF}{\mathrm{F}}

\newcommand{\rL}{\mathrm{L}}

\newcommand{\mob}[1][\mu]{\textsf{Mob}_{#1}}

\newcommand{\sI}{\mathsf{I}}

\newcommand\intl{\int\limits}
\newcommand{\rCS}{\mathrm{CS}}

\newcommand\suc[3][\mu]{\mathrm{Su}^{#3}(#2,#1)}

\newcommand{\noi}{\noindent}

\renewcommand\c{\circ}

\newcommand\md{\,{\mathrm{d}}}
\renewcommand\ge{\geqslant}
\renewcommand\le{\leqslant}

\newcommand{\mI}[1]{\mathbbm{1}_{#1}}

\usepackage{xcolor}\definecolor{darkgreen}{rgb}{0,0.5,0}

\begin{document}
\begin{frontmatter}
\title{Choquet-Sugeno-like operator based on  relation\\ and conditional aggregation operators}

\author{Micha{\l} Boczek\corref{cor1}\fnref{label1}}
\ead{michal.boczek.1@p.lodz.pl}
\cortext[cor1]{michal.boczek.1@p.lodz.pl}

\author{Ondrej Hutn\'{i}k\fnref{label2}}
\ead{ondrej.hutnik@upjs.sk}

\author{Marek Kaluszka\fnref{label1}}
 \ead{marek.kaluszka@p.lodz.pl}

\address[label1]{Institute of Mathematics, Lodz University of Technology, 90-924 Lodz, Poland}
\address[label2]{Institute of Mathematics, Pavol Jozef \v Saf\'arik University in Ko\v sice, 040-01 Ko\v sice, Slovakia}

\begin{abstract}
We introduce a~\textit{Choquet-Sugeno-like operator} generalizing  many operators for bounded functions and monotone measures from the literature, e.g., Sugeno-like operator, Lov\'{a}sz and Owen measure extensions, $\rF$-decomposition integral with respect to a~partition decomposition system, and others. The new operator is based on the concepts of dependence relation and conditional aggregation operators, but it does not depend on $t$-level sets. We also provide conditions for which the  Choquet-Sugeno-like operator coincides with  some Choquet-like integrals defined on finite spaces and appeared recently in the literature, e.g. reverse Choquet integral, $d$-Choquet integral, $\rF$-based discrete Choquet-like integral, some version of $C_{\rF_1\rF_2}$-integral, $\mathrm{C}\mathrm{C}$-integrals (or Choquet-like Copula-based integral) and discrete inclusion-exclusion integral. Some basic properties of the Choquet-Sugeno-like operator are studied.
\end{abstract}


\end{frontmatter}

\section{Introduction}



The origin of investigation of nonadditive integrals goes back to the works of Vitali, Choquet, Shilkret and Sugeno. 
Their works have been extensively studied and several generalizations have been proposed in recent years grouping the mentioned integrals into two main classes: depending on $t$-level set and independent on it. The first group includes universal integral \cite{klement10}, upper and lower $n$-Sugeno integral \cite{BHoHK21}, whereas the second group includes copula-based integrals \cite{klement10,lucca},  $\rF$-decomposition integrals \cite{horanska2020} and inclusion-exclusion integral \cite{Honda2017}. Obviously there are some functionals belonging to both groups such as seminormed fuzzy integrals \cite{suarez} and the prominent nonadditive integrals of Choquet \cite{choquet} and Sugeno \cite{sug74}.
Especially, generalizations of the discrete Choquet integral  have recently attracted the greatest interest.
The discrete Choquet integral can be equivalently expressed as follows 
\begin{align}
(C)\intl_X f\md\mu&=\sum_{i=1}^n \f{i} \cdot(\mu(\bs{i})-\mu(\bs{i+1})),\label{1}
\\&=\sum_{i=1}^n (\f{i}-\f{i-1}) \cdot\mu(\bs{i}), \label{2}
\\&=\sum_{i=1}^n \bigl(\f{i} \cdot \mu(\bs{i})-\f{i-1} \cdot\mu(\bs{i})\bigr),\label{3} 
\\&=\sum_{D\in 2^{X}\setminus\{\emptyset\}} \textsf{Mob}_\mu(D)\cdot \min_{i\in D} f(i), \label{4}
\end{align}
where $X=\{1,2,\dots, n\},$  $\mu$ is a~monotone measure on $2^X,$ integrand $f$ is a~vector with nonnegative entries, 
$(\cdot)\colon X\to X$ is a~permutation such that $0=f_{(0)}\le \f{1}\le\ldots\le \f{n},$ and  $\bs{i} = \{(i),\ldots,(n)\}$ for $i\in \{1,2,\dots,n\}$ with $\bs{n+1}=\emptyset.$ Moreover, $\textsf{Mob}_\mu$ is the M\"{o}bius transform of $\mu.$ 
Replacing the product operation in its standard form~\eqref{1}, equivalent form~\eqref{2}, expanded form~\eqref{3}, and the M\"{o}bius transform form~\eqref{4} by some other fusion functions with appropriate properties 
one can obtain a~resulting aggregation-like function providing various generalizations of the discrete Choquet integral. 
For the present state-of-art of the generalizations of the discrete Choquet integral we refer to~\cite{dimuro2020}. 

As far as we know, there is no unified setting for a~common generalization of all the expressions~\eqref{1}--\eqref{4} of the discrete Choquet integral. In recent paper~\cite{bustince2020} authors write in Conclusion: ``(...) the possible extension of our generalization idea to Choquet integrals expressed in terms of M\"{o}bius transform does not seem so easily achievable, as it is not clear how the information provided by a~restricted dissimilarity function can be included in such representation. Nevertheless, it is worth to mention that, in any case, the equivalence between the different possible representations of the standard Choquet integral will be most probably lost in our more general setting.'' 
Therefore, in this paper we try  to cover some generalizations of expressions~\eqref{1}--\eqref{4} of the discrete Choquet integral in one general formula.
For this purpose we introduce a~Choquet-Sugeno-like operator
(see Definition~\ref{def:choqsug}) independent on $t$-level sets, but depending upon a~conditional aggregation operator and some relation between sets in a~collection. These are two novel ingredients in comparison with the known approaches existing in the literature. Due to these new elements we may provide many well-known and examined operators as well as several new operators not yet studied in the literature. A~detail discussion is included in Examples~\ref{ex:3.3}--\ref{ex:3.8} and in 
Section~\ref{sec:connection}, where we also give relationships between the existing generalizations of discrete Choquet integral and our Choquet-Sugeno-like operator.  As a~by-product, we join the works that simultaneously generalize the Sugeno integral and the Choquet integral on finite sets, similarly to the works \cite{horanska2020,lucca,MesSt19,torra2003}.

The paper is organized as follows. In the forthcoming section  we provide basic notations and definitions we work with. In Section~\ref{sec:choquet}  we give several examples of Choquet-Sugeno-like operator known in the literature.  In the next Section~\ref{sec:connection} we provide the conditions for which our Choquet-Sugeno-like operators coincides with the operators extending the formulas \eqref{1}--\eqref{4}. 
In Section~\ref{sec:property} we examine some basic properties of the Choquet-Sugeno like operator such as monotonicity, homogeneity, subadditivity, convexity and idempotency.

\section{Basic notations}

Let $(X,\Sigma)$ be a~measurable space, where $\Sigma$ is a~$\sigma$-algebra of subsets of $X.$ In what follows, $\Sigma_0=\Sigma\setminus \{\emptyset\}.$
A~{\it monotone} or {\it nonadditive measure} on $\Sigma$ is a~finite nondecreasing set function $\mu\colon \Sigma\to [0,\infty),$ i.e., $\mu(C)\le\mu(D)$ whenever $C\subset D$ with $\mu(\emptyset)=0,$ where ``$\subset$'' and ``$\subseteq$'' denote the proper inclusion and improper inclusion, respectively. 
 We denote the class of all monotone measures  on $(X,\Sigma)$  by $\bM.$ 
If $\mu(X)=1,$ then $\mu$ is called a~\textit{capacity} and $\bM^1$ denotes the set of all capacities. 
For $\mu,\nu\in\bM,$ we write $\mu\le \nu$ whenever $\mu(D)\le \nu(D)$ for any $D\in\Sigma.$ 
For $X=\{1,2,\ldots,n\},$ we say that $\mu\in\bM$ is \textit{symmetric} if the condition $|C|=|D|$ implies $\mu(C)=\mu(D),$ where $|E|$ is the cardinality  of a~set $E.$ Additionally, a~capacity $\mu$ is \textit{symmetric} whenever $\mu(\cdot)=\nu(\cdot)/\nu(X)$ with $\nu\in \bM$ being symmetric. Denote by $\bEM$ the family of all set functions 
$\wmu\colon \Sigma\to (-\infty,\infty)$ with $\wmu(\emptyset)=0.$ 

By $\bF$ denote the set of all $\Sigma$-measurable (measurable, for short) nonnegative bounded functions on $X$ and $\bF^1=\{f\in \bF\colon \textstyle{\sup_{x\in X} f(x)}\le 1\}.$ We write $f\le g$ if $f(x)\le g(x)$ for all $x\in X.$ Increasing [resp. nondecreasing] function  $f\in \bF$ means that $f(x)<f(y)$ [resp. $f(x)\le f(y)$] whenever $x<y.$ We say that a~function $f\in\bF$ is \textit{subadditive} if $f(x+y)\le f(x)+f(y)$ for any $x,y$ such that $x+y\in X.$ For $t\ge 0,$ define the \textit{$t$-level set} of a~function $f\in\bF$ as $\{f\ge t\}=\{x\in X\colon f(x)\ge t\}.$

Let $D_1,D_2\subseteq (-\infty,\infty).$  We say that an operation $\circ\colon D_1\times D_2\to [0,\infty)$ is \textit{nondecreasing} if $a\circ b\le c\circ d$ for any $(a,b), (c,d)\in (D_1,D_2)$ such that $a\le c$ and $b\le d.$  We say that $x\mapsto x\circ b$ is \textit{subadditive} for any $b$ if $(x+y)\circ b\le x\circ b+y\circ b$ for any $b.$ In a~similar way we define the subadditivity of $x\mapsto a\circ x.$

The set $\{1,\ldots,k\}$  is denoted by $[k].$
Let $\mI{D}$ denote the indicator function of a~set $D,$ that is, $\mI{D}(x)=1$ if $x\in D$ and $\mI{D}(x)=0$ otherwise, and $\mI{}(S)$ denote the indicator function of a~logical sentence $S,$ that is, $\mI{}(S)=1$ if $S$ 
is true and $\mI{}(S)=0$ otherwise. 
For any $a,b\in [0,\infty),$ let  $a\wedge b=\min\{a,b\}$ and $a\vee b=\max \{a,b\}$ as well as $(x)_+=x\vee 0$ for any $x\in (-\infty,\infty).$  Moreover, $\mN=\{1,2,\ldots\}$ denotes the set of  natural numbers and $0_D(x)=0$ for all $x\in D,$ where  $D\in\Sigma.$
We adopt the usual conventions: 
$\textstyle{\sum_{k\in\emptyset} f(k)=0}$ and $\textstyle{\sum_{i=j}^k a_i=0}$ for $k<j.$

A~crucial concept used in this paper is an extension of aggregation functions introduced recently in~\cite{BHHK21}. 

\begin{definition}\label{def:conagr}
 A~map  $\aA[D]{\cdot}\colon \bF\to [0,\infty)$   is said to be a~\textit{conditional aggregation operator} with respect to $D\in\Sigma_0$
if it satisfies the following conditions:
\begin{enumerate}[noitemsep]
    \item[$(C1)$]  $\aA[D]{f}\le \aA[D]{g}$  for any $f,g\in\bF$ such that $f(x)\le g(x)$ for all $x\in D$;
    \item[$(C2)$]  $\aA[D]{\mI{D^c}}=0.$
\end{enumerate}
The nonempty set $D$  will be called a~\textit{conditional set}.
\end{definition}  

From Definition~\ref{def:conagr} it follows that $\aA[D]{f}=\aA[D]{f\mI{D}}$ for any $f\in\bF,$ so the value $\aA[D]{f}$ can be interpreted as  ``an  aggregated value of $f$ on  $D$''. In other words, the conditional aggregation operator only depends on the value of the considered function defined on the conditional set.
Conditional aggregation operator extends the concept of aggregation operator $\agg(\cdot)$ presented by Calvo et al. \cite{Calvo} to all measurable functions.
For some examples and methods of construction of conditional aggregation operators we refer to \cite{BHHK21}\footnote{ Observe that in \cite{BHHK21} the conditional aggregation operator can take infinite values as well. To avoid unnecessary complications, we consider here only finite-valued conditional aggregation operators.}. 
By $\bA=\{\aA{\cdot}\colon D\in \Sigma\}$  we denote a~\textit{family of conditional aggregation operators} (FCA in short).
In order for the FCA to be well defined for all the sets from $\Sigma,$ from now on we consider the conditional aggregation operators with the additional assumption $\aA[\emptyset]{\cdot}=0.$  
Several important FCAs $\bA$ will be highlighted using the superscript such as $\bA^{\inf}=\{\aAi{\cdot}{\inf}\colon D\in\Sigma\}$ and $\bA^{\sup}=\{\aAi{\cdot}{\sup}\colon D\in\Sigma\},$ where  $\aAi{f}{\inf}=\textstyle{\inf_{x\in D} f(x)}$ and $\aAi{f}{\sup}=\textstyle{\sup_{x\in D} f(x)}$ for any $D\in\Sigma_0.$
In order to avoid ambiguity in the markings, we still assume that $\aAi[\emptyset]{f}{\inf}=0=\aAi[\emptyset]{f}{\sup}$ when necessary.

\section{Choquet-Sugeno-like operator and its several special cases}\label{sec:choquet}

In this section we introduce an operator which is based on two families of conditional aggregation operators and a~relation between the conditional sets, which is the main ingredient providing new possibilities. Conditional sets will be chosen from a~\textit{collection} being any subset $\cD\subseteq \Sigma_0.$ 
A~nonempty family $\cH$ of collections will be called a~\textit{decomposition system}, i.e., $\cH\subseteq 2^{\Sigma_0}\setminus\{\emptyset\}$ (cf. \cite{horanska2020}). 
Several decomposition systems often used in the paper are summarized in the following example.

\begin{example}\label{ex:3.1}
\begin{enumerate}[noitemsep]
    \item[]
    \item[(a)] $\cH_{\text{one}}=\{\Sigma_0\}$  is a~singleton consisting of the maximal collection; 
    
    \item[(b)]
    $\cH_{\text{part}}=\{\cP\colon \cP \text{ is a~finite partition of } X\}.$ We say that $\cP$ is a~\textit{finite partition} of $X$ if $\cP=\{D_1,\ldots,D_n\}$ such that $\textstyle{\bigcup_{i=1}^n D_i=X}$ and $D_i\cap D_j=\emptyset$ for any  $i\neq j$ and $D_i\in\Sigma_0$ for any $i$; 
    \item[(c)] 
      $\cH_{\text{chain}}=\{\cD_l\colon l\in\mN\},$ where $\cD_l$ 
      is a~\textit{chain} of length $l$ defined as $\cD_l=\{D_l,\ldots,D_1\},$ where $D_l\subset\ldots\subset D_1$ for any $D_i\in\Sigma_0$ and all $i\in [l].$ For instance, for $X=[2]$ we have $\cH_{\text{chain}}=\big\{ \{\{1\}\},\{\{2\}\},\{\{1,2\}\},\{\{1\},\{1,2\}\},\{\{2\},\{1,2\}\}\big\}.$
      \end{enumerate}
\end{example}

We say that $\cR$ is a~\textit{relation on} $\cD\cup \{\emptyset\}$ if $\cR \subseteq (\cD\cup \{\emptyset\})\times (\cD\cup\{\emptyset\}),$ where $\cD$ is a~collection. For two sets $C,D\in\cD\cup\{\emptyset\}$ being in relation $\cR$ we write $(C,D)\in \cR.$ Although the relation $\cR$ depends on a~collection $\cD,$ we will not indicate this dependence explicitly in the notation.

Now we can define the Choquet-Sugeno-like operator. Note that we do not use the word ``integral'', as there is no unambiguous definition of an integral in the literature, see \cite{BHHK21,cattaneo,klement10}. Since this paper is related to another aspects of this topic, we leave this problem for further discussions.

\begin{definition}\label{def:choqsug}
Let $\cH$ be a~decomposition system, $\cD$ a~collection from $\cH,$ and $\cR$
a~relation on $\cD\cup \{\emptyset\}.$ Then for $\rL\colon [0,\infty)^3\times (-\infty,\infty) \to (-\infty,\infty)$  the \textit{Choquet-Sugeno-like operator}
of $f\in\bF,$ $\mu\in\bM$  and $\wmu\in\bEM$  is defined by
\begin{align}\label{n:cs1} 
    \rCS^{\rL}_{\cH,\bA,\widehat{\bA}}(f,\mu,\wmu)=\sup_{\cD\in \cH}\sum_{(C,D)\in \cR}\rL\big(\aA[C]{f}, \waA[D]{f}, \mu(C),\wmu(D)\big),
\end{align}
where $\bA$ and $\widehat{\bA}$ are FCAs.
\end{definition}

\begin{remark}
 Note that it is sufficient to define $\mu$ and $\wmu$  on $\textstyle{\bigcup_{\cD\in\cH} \cD}$ instead of $\Sigma.$
\end{remark}

Let us underline once again that the relation $\cR$ in~\eqref{n:cs1} depends on a~collection $\cD$ although we do not write it. Now we show that the Choquet-Sugeno-like operator generalizes many concepts from the literature.

\begin{example}[upper Sugeno-like operator]\label{ex:3.3}
Let $\rF\colon [0,\infty)^2\to [0,\infty),$ $\cR=\{(D,D)\colon D\in \cD\}$ for 
$\cD\in \cH=\{\{D\}\colon   D\in \Sigma_0\}$ and $\rL(x,y,z,w)=\rF(x,z).$ Then \eqref{n:cs1} takes the form
    \begin{align}\label{n:sug0}
        \rCS^{\rF}_{\cH,\bA}(f,\mu)=    \sup_{D\in \Sigma_0} \rF(\aA[D]{f}, \mu(D)).
    \end{align}
Putting $\aA{\cdot}=\aAi{\cdot}{\inf}$ for any $E\in \Sigma_0$  we obtain the \textit{upper Sugeno-like operator} 
    \begin{align}\label{n:sug1}
        \suc{f}{\rF}=\sup_{D\in \Sigma_0} \rF(\inf_{x\in D} f(x), \mu(D) ).
    \end{align}
In particular, for $\rF\colon [0,1]^2\to [0,1]$ being a~fuzzy conjunction\footnote{A~binary function $\circ\colon [0,1]^2\to [0,1]$ is called a~\textit{fuzzy conjunction}  if it is nondecreasing and fulfils $0\circ 0=0\circ 1=1\circ 0=0$ and $1\circ 1=1.$} 
the operator \eqref{n:sug1} is the  \textit{$q$-integral} \cite{dubois}, whereas for $\rF\colon [0,1]^2\to [0,1]$ being a~semicopula\footnote{A~binary function $\circ\colon [0,1]^2\to [0,1]$ is called a~\textit{semicopula} if it is nondecreasing and fulfils $1\circ a=a\circ 1=a$ for any $a\in [0,1].$} the \textit{seminormed fuzzy integral} \cite{BHH1,BHH2,BHH3,BHH4,suarez} of $(f,\mu)\in \bF^1\times \bM^1$ is recovered. 
The upper Sugeno-like operator with $\rF=\wedge$ is the famous \textit{Sugeno integral} \cite{sug74}, whereas for $\rF=\cdot$ we get the \textit{Shilkret integral} \cite{shilkret}. If $\rF$ is nondecreasing, arguing as in the proof of \cite[Theorem 2.2]{BHH1}, the upper Sugeno-like operator \eqref{n:sug1} can be rewritten in term of $t$-level sets in the following way
\begin{align}\label{n:sug2}
 \suc{f}{\rF}=\sup_{t\ge 0} \rF(t,\mu(\{ f\ge t\}))
\end{align}
known as the \textit{(upper) generalized Sugeno integral} \cite{cattaneo,KOB_Chebyshev}. 
For $X=[n]$ the operator \eqref{n:sug2} takes the form
\begin{align}\label{n:sug3}
    \suc{f}{\rF}=\max_{i\in [n]} \rF (\f{i}, \mu(\bs{i})),
\end{align}
where  $(\cdot)\colon [n]\to [n]$ is a~permutation such that $\f{1}\le\ldots\le \f{n}$ and $\bs{i}= \{(i),\ldots,(n)\}$ for $i \in [n],$ as $\rF$ is nondecreasing.
The operator presented in  \eqref{n:sug3} was studied by Horansk\'{a} and \v{S}ipo\v{s}ov\'{a} \cite{horanska2018} for  $\Sigma=2^{[n]},$  $\rF\colon [0,1]^2\to [0,1]$  and $(f,\mu)\in \bF^1\times \bM^1.$ 

\end{example}

\begin{example}[generalized Lebesgue integral for sum] \label{ex:3.4}
Let $\rL(x,y,z,w)=\rF(x,z).$
Then \eqref{n:cs1} with $\bA=\bA^{\inf}$ and $\cR=\{(D,D)\colon D\in \cD\}$ for $\cD\in\cH$ can be rewritten as follows 
    \begin{align}\label{n:leb1}
        \rCS^{\rF}_{\cH,\bA^{\inf}}(f,\mu)=\sup_{\cD\in\cH} \sum_{D\in\cD} \rF\big(\aAi[D]{f}{\inf},\mu(D)\big).
    \end{align}
Putting $\rF(x,z)=x\otimes z$ and  $\cH=\cH_{\text{part}}$ in \eqref{n:leb1} we get 
    \begin{align}\label{n:leb2}
        \rCS^{\otimes}_{\cH_{\text{part}},\bA^{\inf}}(f,\mu)=\sup_{ \cD\in \cH_{\text{part}}}\sum_{D\in \cD} \aAi[D]{f}{\inf}\otimes \mu(D)
    \end{align}
the \textit{generalized Lebesgue integral with $\oplus = +$}  defined in \cite[Definition 3.1]{zhang}, where $\otimes$ is a~pseudo-multiplication (see \cite[Definition 2.2]{zhang}). 
Let $\otimes=\cdot$ and $X=(a,b],$ where $a<b$ and $a,b\in (-\infty,\infty).$ Consider
$$
\cH_{\text{part}}^\ast =\{ \{I_1,\ldots,I_k\}\colon k\in \mN\}\subset \cH_{\text{part}},
$$
where $I_i=(x_i,x_{i+1}]$ with the measure $\mu_G(I_i)=G(x_{i+1})-G(x_i)$ for 
a~nondecreasing function $G$ on $[a,b].$ Then \eqref{n:leb2} is the \textit{lower Darboux-Stieltjes integral}  of the form 
    \begin{align*}
        \rCS^{\cdot}_{\cH_{\text{part}}^{\ast},\bA^{\inf}}(f,\mu_G)=\sup_{\cD\in \cH_{\text{part}}^{\ast}}\sum_{I_i\in \cD} \mu_G(I_i)\cdot \inf_{x\in I_i} f(x).
    \end{align*}
Setting $\cH=\cH_{\text{part}}^{\ast\ast}=\{\{X\}\}$ and $\rF(x,z)=x\cdot z$ in \eqref{n:leb1} we get the \textit{min-max integral} \cite[Theorem 4.13]{seliga}     
$$
\rCS^{\cdot}_{\cH_{\text{part}}^{\ast\ast},\bA^{\inf}}(f,\mu) =  \mu(X)\cdot\inf_{x\in X} f(x).
$$
\end{example}


\begin{example}[$\rF$-decomposition integral with respect to $\cH_{\text{part}}$]\label{ex:3.5}

Let $X=[n]$ and $\Sigma=2^{[n]}.$
Observe that the operator \eqref{n:leb1} with $\cH=\cH_{\text{part}}$  and a~nondecreasing function $\rF$ takes the form
   \begin{align}\label{n:pan1}
   \rCS^{\rF}_{\cH_{\text{part}},\bA^{\inf}}(f,\mu)= \mathcal{I}^{\rF}_{\cH_{\text{part}},\mu}(f),
    \end{align}
where
    \begin{align}\label{n:pan2}
    \mathcal{I}^{\rF}_{\cH,\mu}(f)=\sup\Big\{ \sum_{D\in\cD} \rF (a_D,\mu(D))\colon \sum_{D\in \cD} a_D\mI{D}\le f\,\,\cD\in\cH\Big\}
    \end{align}
is an \textit{$\rF$-decomposition integral}. 
Setting $\rF=\cdot$ in \eqref{n:pan1}  we obtain the \textit{Pan-integral} based on the standard arithmetic operation addition and multiplication \cite{yang}.
The operator \eqref{n:pan2} is studied in \cite{horanska2020} for $\rF\colon [0,1]^2\to [0,\infty)$ and $(f,\mu)\in \bF^1\times \bM^1.$

The assumption  $\cH=\cH_{\text{part}}$ in \eqref{n:pan1} is essential to prove the equivalence between formula  \eqref{n:leb1} and \eqref{n:pan1}  for any nondecreasing binary function $\rF$ and any $(f,\mu)\in \bF\times\bM.$ Indeed, let $X=[3]$ and $\cH=\{ \{D_1,D_2,D_3\}\}\neq \cH_{\text{part}},$ where  $D_1=\{1\},$ $D_2=\{1,3\}$ and $D_3=X.$ Assume that $\mu(C)=1$ for any $\emptyset\neq C\subseteq X,$ $f(1)=0.4,$ $f(2)=0.2$ and $f(3)=0.3.$ Then \eqref{n:leb1} with $\rF(x,y)=xy$ takes the form $\rCS^{\cdot}_{\cH,\bA^{\inf}}(f,\mu)=\textstyle{\sum_{i=1}^3 \aAi[D_i]{f}{\inf}} =0.9.$  On the other hand, \eqref{n:pan2} has the form 
    \begin{align*}
       \mathcal{I}^{\cdot}_{\cH,\mu}(f)&= \sup\Big\{\sum_{i=1}^3 a_{D_i}\colon \sum_{i=1}^3 a_{D_i}\mI{D_i}\le f\Big\}
       \le f(1),
    \end{align*}
since the condition $\textstyle{\sum_{i=1}^3 a_{D_i}\mI{D_i}\le f}$ implies that $a_{D_1}+a_{D_2}+a_{D_3}\le f(1),$ and so $\rCS^{\cdot}_{\cH,\bA^{\inf}}(f,\mu)>\mathcal{I}^{\cdot}_{\cH,\mu}(f).$
\end{example}
    

\begin{example}[generalization of the Lov\'{a}sz extension]\label{ex:3.6}
Let $X=[n]$ and  $\Sigma=2^{[n]}.$ Assume that $\cR=\{(D,D)\colon D\in \cD\}$ for $\cD\in\cH_{\text{one}},$ 
$\rL(x,y,z,w)=y\circ w$ and $\wmu(D)=\mob(D)$ with $\mu\in\bM,$ 
where $\mob(D)=\textstyle{\sum_{C\subseteq D} (-1)^{|D\setminus C|} \mu(C)}$ is the~M\"obius transform (\cite[Section 2.10]{grabisch2016}). Then \eqref{n:cs1} is  a~\textit{generalization of the Lov\'asz extension},
\begin{align}\label{n:chm1}  
\rCS^{\circ}_{\cH_{\text{one}},\bA}(f,\mob{})=\sum_{\emptyset \neq D\subseteq X} \aA[D]{f} \circ \mob (D).
\end{align}
In particular, for $\circ=\cdot$ and $\aA[D]{\cdot}=\aAi[D]{\cdot}{\inf}$ we get the \textit{discrete Choquet integral} expressed in terms of the M\"obius transform, known as the \textit{Lov\'asz measure extension} \cite{lovasz}. On the other hand, for the product conditional aggregation operator $\aAi[D]{f}{\textrm{prod}}=\textstyle{\prod_{i\in D}f(i)},$ the formula~\eqref{n:chm1} is the generalized Owen extension of $\mu.$ The original Owen measure extension corresponds to $\sA^{\textrm{prod}}(\cdot|\cdot)$ and $\c=\cdot,$ see~\cite{Owen}. Assume that $\circ\colon [0,1]\times  (-\infty,\infty)\to (-\infty,\infty)$ is bounded on $[0,1]^2,$   $(f,\mu)\in \bF^1\times \bM^1$ and $\aA{h}=\mathsf{Agg}(h \mI{E})$ is an aggregation operator (see \cite{Calvo}). Then the operator \eqref{n:chm1} with
\begin{itemize}[noitemsep]
    \item $\circ=\cdot$ is studied by Koles\'{a}rov\'{a} et al. \cite{kol2012},
    \item $\mathsf{Agg}(h \mI{E})=\textstyle{\inf_{x\in E} h(x)}$ is examined by Fernandez et al. \cite{fernandez20},
    \item $\circ$ with the values in $[0,1]$ is discussed by Horansk\'a \cite{horanska20a}.
\end{itemize}
\end{example}

\begin{example}\label{ex:3.7}
Let $X=[n]$ and $\Sigma=2^{[n]}.$ Put $\cR=\{(D,D^c)\colon D\in \cD\}$ for $\cD\in\cH_{\text{one}},$  
$\rL(x,y,z,w)=(x-y)_+\circ z,$ $\bA=\bA^{\inf}$ and $\widehat{\bA}=\bA^{\sup}.$ Then \eqref{n:cs1} takes the form
\begin{align}\label{n:r1}
    \rCS^{\rL}_{\cH_{\text{one}},\bA^{\inf},\bA^{\sup}}(f,\mu)=\sum_{\emptyset \neq D\subseteq X} (\min_{x\in D} f(x)-\max_{x\in D^c} f(x))_+\circ \mu(D).
\end{align}
For $\circ=\cdot\,,$ we get an alternative representation of the Choquet integral presented in \cite{Jin2018}, see also \cite[p.\,149]{BBC}. 
\end{example}

\begin{example}\label{ex:3.8} (Generalized $p$-variation) 
Let $\cR \subseteq \cD\times \cD$  
and $\rL(x,y,z,w)=|x-y|^p,$ where $p\ge 1.$ 
Then the operator \eqref{n:cs1} with $\bA=\widehat{\bA},$ called  the {\it generalized $p$-variation}, takes the form
\begin{align*}
    {\mathrm V}^p_{\bA}(f)=\sup_{\cD\in \cH}\sum_{(C,D)\in \cR} \big|\aA[C]{f}- \aA[D]{f}\big|^p.
\end{align*}
The well known notion of $p$-variation we get when taking $X=[a,b],$ $\Sigma=2^X,$ $\aA[\{x\}]{f}=f(x)$ and $\cR=\{(\{x_{i-1}\},\{x_i\})\colon \{x_i\}\in \cD,\, i\in [n]\},$ where  $\cD\in \cH=\{\{\{x_0\},\{x_1\},\ldots,\{x_k\}\}\colon a=x_0<x_1<\ldots<x_{n-1}<x_{n}=b,\,n\in\mN\}.$
\end{example}

\begin{remark}
General properties of operators defined in \eqref{n:sug0}, \eqref{n:leb1}, \eqref{n:chm1} and \eqref{n:r1} have not 
been studied in the literature so far.
\end{remark}

\section{Connections with other operators with respect to monotone measure: finite space case}\label{sec:connection}

In \cite{bustince2020, horanska2018,horanska2020, klement10,lucca,lucca2018,lucca2019,meng,MesK} the authors study properties of operators \eqref{1}--\eqref{3}
mainly by replacing the product by another binary function. So, it is natural to find a~connection between them and the Choquet-Sugeno-like operator. We will describe them in this section.  Additionally, we provide a~relationship between the Choquet-Sugeno-like operator and the discrete inclusion-exclusion integral \cite{Honda2017}.

\medskip 

In this section we assume that $X=[n]$ and $\Sigma=2^{[n]}$ with $n\ge 2.$
Moreover, to shorten the notation, we
introduce two relations $\cR^+$ and $\cR^-$ on $\cD_l\cup \{\emptyset\}$ (see Example~\ref{ex:3.1}\,(c)) as follows:
    \begin{itemize}[noitemsep]
        \item $\cR^+=\big\{(D_1,D_{2}),(D_2,D_3),\ldots,(D_{l},D_{l+1})\big\}$;
        \item $\cR^-=\big\{(D_1,D_{0}),(D_2,D_1),\ldots,(D_{l},D_{l-1})\big\},$
    \end{itemize}
where $D_l\subset D_{l-1}\subset \ldots\subset D_1$ and $D_{0}=D_{l+1}=\emptyset.$

\subsection{\textbf{Connection with $\fC_{\circ}^{\mu,\wmu}$}}

Let $\circ\colon [0,\infty)^2\to [0,\infty).$
The operator $\fC_{\circ}^{\mu,\wmu}$ is defined as follows
    \begin{align}\label{n:c1a}
     \fC_{\circ}^{\mu,\wmu}(f)= \sum_{i=1}^n \f{i}\circ (\mu(\bs{i})-\wmu(\bs{i+1}))
    \end{align}
for any $(f,\mu,\wmu)\in \bF\times \bM\times \bM$ such that $\mu\ge\wmu,$ where $(\cdot)\colon [n]\to [n]$ is a~permutation such that $\f{1}\le\ldots\le \f{n},$ $\bs{i}= \{(i),\ldots,(n)\}$ for $i \in [n]$ and $\bs{n+1}=\emptyset.$ 
The operator \eqref{n:c1a} depends on the permutation, which need not be unique in general. Therefore, in order the operator $\fC_{\circ}^{\mu,\wmu}$ to be well defined, its evaluation must not depend on the considered permutation.
Observe that if $f$ is either decreasing  or increasing, then the~permutation is unique. This means that in the class of all decreasing  or increasing functions, the operator $\fC_{\circ}^{\mu,\wmu}$  is well defined for any binary operation $\circ$ and $\mu,\wmu\in\bM.$ 
If we assume that  the map $x\mapsto a\circ x$ is Lebesgue measurable for any $a,$ then using  the same argument as Horansk\'a and \v{S}ipo\v{s}ov\'a \cite[Proposition 12]{horanska2018} one can show that the operator \eqref{n:c1a} is well defined for any $(f,\mu,\wmu)\in \bF\times \bM\times \bM$ if and only if  $a\circ b=g(a)b$ for some function $g\colon [0,\infty)\to [0,\infty).$ 
Also the special case of \eqref{n:c1a} with $\circ\colon [0,1]^2\to [0,1],$ $f\in \bF^1$ and $\mu=\wmu\in \bM^1$ is examined in \cite{horanska2018}.

\begin{proposition}\label{pro:4.1}
The operator \eqref{n:c1a} is well defined for any $f\in\bF$ and any binary operation $\circ\colon [0,\infty)^2\to [0,\infty)$ whenever $\mu,\wmu\in\bM$ are symmetric such that $\mu\ge \wmu.$
\end{proposition}

In terms of conditional aggregation operators the operator \eqref{n:c1a} can be rewritten as
\begin{align}\label{n:c1ab}
    \fC_{\circ}^{\mu,\wmu}(f)= \sum_{i=1}^n \aAi[\bs{i}]{f}{\inf}\circ (\mu(\bs{i})-\wmu(\bs{i+1})).
\end{align}
This form resembles the Choquet-Sugeno-like operator \eqref{n:cs1} with $\cH=\cH_{\text{chain}},$ $\cR=\cR^+,$ 
$\bA=\bA^{\inf},$ $\mu,\wmu\in\bM$ such that $\mu\ge \wmu$ and $\rL(x,y,z,w)=x\circ (z-w)_+,$  i.e.,
    \begin{align}\label{n:c1b}
        \rCS^{\rL}_{\cH_{\text{chain}},\bA^{\inf}}(f,\mu,\wmu)      =\sup_{\cD\in\cH_{\text{chain}}} \sum_{i=1}^l \aAi[D_{i}]{f}{\inf} \circ (\mu(D_{i})-\wmu(D_{i+1}))
    \end{align}
    under the convention that $D_{l+1}=\emptyset.$
However, the following example demonstrates that the operators \eqref{n:c1ab} and \eqref{n:c1b} are not the same.

\begin{example}\label{ex:4.2}
Put $X=[2],$ $f(1)=0.5,$ $f(2)=1,$  $\mu(X)=1,$ $\mu(\{1\})=0.5$ and  $\mu(\{2\})=0.4.$ 
By \eqref{n:c1a} with $\circ=\wedge$ and $\wmu=\mu$ we get
    \begin{align*}
    \fC^{\mu,\mu}_{\wedge}(f)&=\sum_{i=1}^2 \f{i}\wedge (\mu(\bs{i})-\mu(\bs{i+1}))
    \\&=0.5\wedge (\mu(X)-\mu(\{2\}))+1\wedge \mu(\{2\})=0.9.
    \end{align*}
    Considering $\rL(x,y,z,w)=x\wedge (z-w)_+$ and the chain  $\cD_2=\{D_2,D_1\}$ with $D_2=\{1\}$ and $D_1=X$  we obtain
    \begin{align*}
       \rCS^{\rL}_{\cH_{\text{chain}},\bA^{\inf}}(f,\mu,\mu)&\ge \aAi[X]{f}{\inf} \wedge (\mu(X)-\mu(\{1\}))+\aAi[\{1\}]{f}{\inf} \wedge \mu(\{1\})=1>\fC^{\mu,\mu}_{\wedge}(f).
    \end{align*}
In consequence, operators \eqref{n:c1ab} and \eqref{n:c1b} are different.
\end{example}

\noi Reason for the non-equivalence between operators \eqref{n:c1ab} and \eqref{n:c1b} is the lack of symmetry of monotone measure.
The following result provides a~necessary and sufficient condition under which both operators are equal.

\begin{tw}\label{tw:4.3}
Assume that $\circ\colon [0,\infty)^2\to [0,\infty)$ is nondecreasing. 
Then  the operator defined in  \eqref{n:c1b} coincides with $\fC^{\mu,\wmu}_\circ (f)$  for any $f\in\bF$ and any symmetric $\mu,\wmu\in\bM$ such that $\mu\ge \wmu$ if and only if  the function $x\mapsto a\circ x$ is subadditive for any $a.$ 
\end{tw}
\begin{proof}
The value of $\fC^{\mu,\wmu}_{\circ}(f)$ for any $f\in \bF$ does not depend on permutations, since $\mu,\wmu$ are symmetric (see Proposition~\ref{pro:4.1}). \\
\noi  ``$\Rightarrow$''    Let  
$f=b\mI{\{2,\ldots,n\}},$ where $b\ge 0.$ Clearly, 
$$
\fC^{\mu,\mu}_\circ (f)=0\circ (\mu(\bs{1})-\mu(\bs{2}))+\sum _{i=2}^n b\circ (\mu(\bs{i})-\mu (\bs{i+1}))
$$
for any symmetric $\mu\in\bM,$ where $\bs{n+1}=\emptyset.$ 
The operator \eqref{n:c1b} coincides with $\fC^{\mu,\mu}_{\circ},$  so  for the chain $\cD_2=\{D_2,D_1\}$ with $D_2=\bs{2}$ and $D_1=\bs{1}$ we get
$$
0\circ (\mu(\bs{1})-\mu(\bs{2}))+b\circ \mu(\bs{2})\le \fC^{\mu,\mu}_\circ (f)
$$
for any symmetric $\mu\in\bM.$
In consequence, we have
\begin{align*}
    b\circ \sum_{i=2}^n (\mu(\bs{i})-\mu(\bs{i+1}))\le \sum _{i=2}^n b\circ (\mu(\bs{i})-\mu (\bs{i+1}))
\end{align*}
for any symmetric $\mu\in\bM.$  This implies that 
$x\circ (y+z)\le x\circ y+x\circ z$ for any $x,y,z.$

 \noi  ``$\Leftarrow$''
Since 
$\{\bs{n},\ldots,\bs{1}\}$ is a~chain,  by \eqref{n:c1ab}  we have
  \begin{align*}
   \sup_{\cD\in\cH_{\text{chain}}} \sum_{i=1}^l \aAi[D_{i}]{f}{\inf} \circ (\mu(D_{i})-\wmu(D_{i+1}))\ge \fC^{\mu,\wmu}_\circ (f).
   \end{align*}
   To prove the statement we will show that
\begin{align*}
    L&:=\sum_{i=1}^l \aAi[D_{i}]{f}{\inf} \circ (\mu(D_{i})-\wmu(D_{i+1}))  \le \fC^{\mu,\wmu}_\circ (f)
\end{align*}
    for any $\cD_l\in\cH_{\text{chain}}$ such that $l\in [n].$
    For the sake of clarity of presentation, we will consider separately the case of a~chain consisting of one nonempty set.

 \begin{itemize}
   \item Let $\cD_1=\{D_1\}\in\cH_{\text{chain}}.$    Then $L=\fs{D_1} \circ \mu(D_1),$  as $D_2=\emptyset,$ due to the  convention. Since $X$ is a~finite set, so $\{k\colon \fs{D_1}=f_{(k)}\}\neq\emptyset.$ Let $k_1=\min\{k\colon \fs{D_1}=f_{(k)}\}.$ Thus $D_1\subseteq \bs{k_1}$ and  by monotonicity of $\circ$ and subadditivity of $x\mapsto a\circ x$ we get
    \begin{align*}
        L&\le \f{k_1} \circ \mu(\bs{k_1})=\f{k_1}\circ \big( \mu(\bs{k_1})-\wmu(\bs{k_1+1})+\wmu(\bs{k_1+1})\big)
        \\&\le  \f{k_1}\circ ( \mu(\bs{k_1})-\wmu(\bs{k_1+1}))+\f{k_1}\circ\wmu(\bs{k_1+1}).
      \end{align*}
    Since $(\f{i})_i$ is a~nondecreasing sequence and $\wmu\le \mu,$ we get
    \begin{align*}
        L&\le \f{k_1}\circ ( \mu(\bs{k_1})-\wmu(\bs{k_1+1}))+\f{k_1+1}\circ\mu(\bs{k_1+1})
        \\&\le \ldots\le \sum_{i=k_1}^n \f{i}\circ ( \mu(\bs{i})-\wmu(\bs{i+1}))\le \fC^{\mu,\wmu}_\circ (f).
    \end{align*}
    \item Let $\cD_l\in \cH_{\text{chain}}$ for fixed $2\le l\le n.$   Then $L=S_1,$ where
        \begin{align*}
         S_k=\sum_{i=k}^l\fs{D_i} \circ (\mu(D_i)-\wmu(D_{i+1}))
        \end{align*}
   for $k\in [l]$ with the convention $D_{l+1}=\emptyset.$ 
   Let $k_1=\min\{k\colon \fs{D_1}=f_{(k)}\}.$ Then $D_1\subseteq \bs{k_1}$ and 
         \begin{align}\label{n:c1d}
            L&\le\f{k_1}\circ (\mu(\bs{k_1})-\wmu(D_2))+S_2.
         \end{align}
    Note that $D_2\subset \bs{k_1}$ and $|D_2|<|D_1|.$ Thus $|D_2|<|\bs{k_1}|$ and $|D_2|\le |\bs{k_1+1}|,$ as $|\bs{k_1+1}|+1=|\bs{k_1}|.$ Then there exists $k_2> k_1$  such that $|D_2|=|\bs{k_2}|$ and $\fs{D_2}\le \fs{\bs{k_2}}=\f{k_2}.$ Indeed, if $D_2=\bs{k_2},$ then $\fs{D_2}=\fs{\bs{k_2}}.$  If $D_2\neq \bs{k_2},$ then there exists $j$ such that $j\in D_2$ and $j\notin \bs{k_2}$  as $|D_2|=|\bs{k_2}|.$
    Observe that $\bs{k_1}=\{(k_1),\ldots,(k_2-1)\}\cup \bs{k_2}$ and $f(i)\le \fs{\bs{k_2}}$ for any $i\in \{(k_1),\ldots,(k_2-1)\},$ 
   so    $f(j)\le \fs{\bs{k_2}},$ as $j\in \bs{k_1}\setminus \bs{k_2}.$ In consequence  $\fs{D_2}\le \fs{\bs{k_2}}.$ 
    By \eqref{n:c1d}  and the symmetricity of $\mu$ and $\wmu,$ we have
    \begin{align*}
        L&\le \f{k_1}\circ (\mu(\bs{k_1})-\wmu(\bs{k_2}))+\f{k_2}\circ (\mu(\bs{k_2})-\wmu(D_3))+S_3.
    \end{align*}
    Using  $(k_2-k_1-1)$ times monotonicity of $\circ$ and subadditivity of $x\mapsto a\circ x,$  we get
    \begin{align*}
        L&\le  \sum_{i=k_1}^{k_2-1} \f{i}\circ (\mu(\bs{i})-\wmu(\bs{i+1}))+\f{k_2}\circ (\mu(\bs{k_2})-\wmu(D_3))+S_3.
    \end{align*}
    Repeating the same arguments for $D_3,\ldots,D_{l}$ leads to the desired inequality 
    \begin{align*}
    L&\le  \sum_{i=k_1}^{k_{l}-1} \f{i}\circ (\mu(\bs{i})-\wmu(\bs{i+1}))+\f{k_l}\circ \mu(\bs{k_{l}})\le \fC^{\mu,\wmu}_\circ (f).   \end{align*}
\end{itemize} 
This completes the proof.
\end{proof}

\begin{remark}
The assumption on symmetricity of $\mu\in\bM$ cannot be omitted (see Example~\ref{ex:4.2}, where $x\mapsto a\wedge x$ is subadditive for any $a\ge 0$).
\end{remark}

In a~similar way to the proof of Theorem~\ref{tw:4.3} one can show that if $\circ\colon [0,1]^2\to [0,1]$ is nondecreasing, then  the operator defined in  \eqref{n:c1b} coincides with $\fC^{\mu,\wmu}_\circ (f)$  for any $f\in\bF^1$ and any symmetric  capacities $\mu,\wmu$ such that $\mu\ge \wmu$ if and only if  the function $x\mapsto a\circ x$ is subadditive for any $a\in [0,1].$  
  The symmetric capacities $\mu$ appear in the natural way when dealing with the order statistics \cite{hoeffding,renyi} from discrete probability distributions. This is due to the fact that the operator \eqref{n:c1a} with $\circ=\cdot$ and symmetric capacity $\mu=\wmu$ is the Choquet integral being a~generalization of the OWA operator\footnote{R.R. Yager in \cite{yager} defined the ordered weighted averaging  (OWA) operator as an expected value of order statistics for random variables defined in finite possible outcomes.}. 

\medskip
Now we show that the operator $\fC^{\mu,\mu}_{\circ}$ can be rewritten in terms of $t$-level sets.

\begin{proposition}\label{pro:4.6}
If  $a\circ b=g(a)b$ for some nonnegative function $g$ with $g(0)=0,$
then 
\begin{align}\label{n:c1e}
    \fC_{\circ}^{\mu,\mu}(f) = \sum_{i=1}^n \f{i} \circ \big( \mu(\{f\ge \f{i}\})-\mu(\{f\ge \f{i+1}\})\big)
\end{align}
holds for any $(f,\mu)\in\bF\times \bM,$ where $(\cdot)$ is a~permutation such that $\f{1}\le\ldots\le \f{n}$  under the convention  $\{f\ge \f{n+1}\}=\emptyset.$
\end{proposition}
\begin{proof}
If  $(\f{i})_{i=1}^n$  is an increasing sequence, the statement is obvious. Assume that $\f{1}<\ldots <\f{j}=\ldots=f_{(k)}<\f{k+1}<\ldots<\f{n}$ for some $j<k.$ Then by \eqref{n:c1a} and Abel transformation\footnote{\textit{Abel transformation}: $\textstyle{\sum_{i=1}^n a_i (b_i-b_{i+1})=\sum_{i=1}^n b_i (a_i-a_{i-1})}$ under the convention $b_{n+1}=a_0=0.$} we get
\begin{align*}
    \fC_{\circ}^{\mu,\mu}(f)& 
    = \sum_{i=1}^n (g(\f{i})-g(\f{i-1})) \mu(\bs{i}),
    \end{align*}
where $B_{(n+1)}=\emptyset,$ $g(0)=0$ and $f_{(0)}=0.$ Moreover, $\bs{i}=\{f\ge \f{i}\}$ for any $i\notin\{j+1,\ldots,k\}.$ As $g(f_{(j)})=\ldots=g(f_{(k)}),$ we get
\begin{align*}
    \fC_{\circ}^{\mu,\mu}(f)&=  \sum_{i=1}^j (g(\f{i})-g(\f{i-1})) \mu(\{f\ge \f{i}\})+ \sum_{i=k+1}^n (g(\f{i})-g(\f{i-1})) \mu(\{f\ge \f{i}\})
    \\&= \sum_{i=1}^n (g(\f{i})-g(\f{i-1})) \mu(\{f\ge \f{i}\})   =\sum_{i=1}^n g(\f{i})\big(\mu(\{f\ge \f{i}\})-\mu(\{f\ge \f{i+1}\})\big),
    \end{align*}
where in the last equality the Abel transformation has been used again. The other cases can be treated similarly, so we omit them.
\end{proof}

It is clear that \eqref{n:c1e} is true for any $\mu\in\bM$ and any binary operation $\circ$ whenever $f\in\bF$ is decreasing or increasing. However, it is not possible to obtain \eqref{n:c1e} for any $f\in\bF$ and any binary operation $\circ$ even if $\mu\in\bM$ is symmetric as it is illustrated in the following example.

\begin{example}\label{ex:4.6}
Let $X=[2],$ $f(1)=f(2)=0{.}5$ and $\mu\in\bM$ be symmetric such that $\mu(X)=2$ and $\mu(\{2\})=1=\mu(\{1\}).$ The operator $\fC_{\wedge}^{\mu,\mu}$ does not depend on permutation, so for $\bs{1}=X$ and $|\bs{2}|=1$ we get
\begin{align*}
        \fC_{\wedge}^{\mu,\mu}(f)=0{.}5\wedge (\mu(X)-\mu(\bs{2})) + 0{.}5\wedge \mu(\bs{2})=1.
    \end{align*}
The right-hand side of \eqref{n:c1e} with $\circ=\wedge$ takes the form
    \begin{align*}
        \sum_{i=1}^2 \f{i} \wedge \big( \mu(\{f\ge \f{i}\})-\mu(\{f\ge \f{i+1}\})\big)=0{.}5\wedge 0+0{.}5\wedge 2=0{.}5<\fC_{\wedge}^{\mu,\mu}(f).
    \end{align*}
To sum up, the equality \eqref{n:c1e} does not hold.
\end{example}

\subsection{\textbf{Connection with $\fR\fC^{\mu,\wmu}_{\circ}$}} 

Let $\circ\colon [0,\infty)^2\to [0,\infty).$
Define the operator $\fR\fC^{\mu,\wmu}_{\circ}$ as follows
    \begin{align}\label{n:c5a}
     \fR\fC^{\mu,\wmu}_{\circ} (f)= \sum_{i=1}^n \f{i}\circ (\mu(\cs{i})-\wmu(\cs{i-1}))
    \end{align}
for any $(f,\mu,\wmu)\in \bF\times \bM\times \bM$ such that $\mu\ge\wmu,$  where $(\cdot)\colon [n]\to [n]$ is a~permutation such that $\f{1}\le\ldots\le \f{n}$ and $\cs{i}= \{(1),\ldots,(i)\}$ for $i \in [n]$ with $\cs{0}=\emptyset.$ 
The operator \eqref{n:c5a} is well defined if $f$ is monotone (decreasing or increasing), $\mu,\wmu\in\bM$ are symmetric or $a\circ b=g(a)b$ for some $g.$ Putting $\mu=\wmu\in\bM^1$ and $\circ=\cdot$ in \eqref{n:c5a}, we obtain the \textit{reverse Choquet integral} introduced in \cite{meng}.

Observe that if $\mu([n])=\wmu([n]),$ $\nu(D)=\wmu([n])-\wmu(D^c)$ 
and $\wnu(D)=\mu([n])-\mu(D^c)$ for any $D\in 2^{[n]},$ then 
\begin{align}\label{n:c5b}
    \fR\fC^{\mu,\wmu}_{\circ} (f)&=\fC^{\nu,\wnu}_{\circ}(f).
\end{align}
Hence, Theorem~\ref{tw:4.3}  can also be used to analyze the relationship between the special case of Choquet-Sugeno-like operator and ${\mathfrak R}\fC_{\circ}^{\mu,\wmu}.$

\medskip
\subsection{\textbf{Connection with $\fC^{\circ,\delta}_\mu$}}

\begin{definition}\label{def:dis}
A~function $\delta\colon [0,\infty)^2\to [0,\infty)$ is said to be a~\textit{dissimilarity function} if for all $x,y,z$ the following conditions hold:
\begin{enumerate}[noitemsep]
    \item[(a)] $\delta(x,y)=\delta(y,x)$;
    \item[(b)] $\delta(x,y)= 0$ if and only if  $x= y$;
    \item[(c)] if $x\le y\le z,$ then $\delta(x,y)\le \delta(x,z)$ and $\delta(y,z)\le \delta(x,z).$
    \end{enumerate}
\end{definition}

Let $\circ\colon [0,\infty)^2 \to [0,\infty)$ and $\delta$ be a~dissimilarity function.
Now we define the following operator 
  \begin{align}\label{n:d1a}
    \fC_{\mu}^{\circ,\delta}(f)=\sum_{i=1}^n \delta(\f{i},\f{i-1})\circ\mu(\bs{i})
\end{align}
for any $(f,\mu)\in \bF\times \bM,$ where  $(\cdot)\colon [n]\to [n]$ is a~permutation such that $\f{1}\le\ldots\le \f{n},$  $\bs{i}= \{(i),\ldots,(n)\}$ for $i \in [n]$ under the convention $\f{0}=0.$ 
In order the operator \eqref{n:d1a} would not depend on the permutation it is necessary to assume that
either $0\circ b=0$ for all $b$ (cf. \cite{MesK}) or $f\in\bF$ is monotone (decreasing or increasing) or $\mu\in\bM$ is symmetric. 

The operator \eqref{n:d1a} with $\circ=\cdot,$  $(f,\mu)\in\bF^1\times \bM^1$ and  $\delta\colon [0,1]^2\to [0,1]$ being a~restricted dissimilarity function\footnote{
A~function $\delta\colon [0,1]^2\to [0,1]$ is said to be a~\textit{restricted dissimilarity function} if it satisfies, for all $x,y,z\in [0,1],$ the conditions (a), (b) and (c) from Definition~\ref{def:dis} and $\delta(x,y)=1$ if and only if $\{x,y\}=\{0,1\}$ (see \cite[Definition 2.1]{bustince2020}).} is the \textit{$d$-Choquet integral} defined in \cite{bustince2020}. 
Study of \eqref{n:d1a}  for  $\circ\colon  [0,1]^2\to [0,1],$ $(f,\mu)\in \bF^1\times \bM^1$ and $\delta(x,y)=|x-y|,$ i.e., 
 \begin{align*}
    \fC_{\mu}^{\circ}(f)=\sum_{i=1}^n (\f{i}-\f{i-1})\circ\mu(\bs{i}),
\end{align*}
was conducted by Mesiar et al. \cite{MesK} and call it as \textit{$\circ$-based discrete Choquet-like integral}.
Later, it was redefined by Lucca et al. \cite{lucca2018} in the way $\overline{\fC}_{\mu}^{\circ}(f)=\min\{1,\fC_{\mu}^{\circ}(f)\}$ for any $(f,\mu)\in \bF^1\times \bM^1$ and was used in a~fuzzy rule-based classification system.

\medskip
Similarly as earlier, we give a~relationship between \eqref{n:d1a} and the Choquet-Sugeno-like operator \eqref{n:cs1} with $\cH=\cH_{\text{chain}},$ $\cR=\cR^-,$ $\rL(x,y,z,w)=\delta(x,y)\circ z$ and $\bA=\bA^{\inf}=\widehat{\bA},$ i.e.,
    \begin{align}\label{n:d1b}
        \rCS^{\rL}_{\cH_{\text{chain}},\bA^{\inf},\bA^{\inf}}(f,\mu)=\sup_{\cD\in \cH_{\text{chain}}}\sum_{i=1}^l \delta\big(\aAi[D_i]{f}{\inf},\aAi[D_{i-1}]{f}{\inf}\big)  \circ \mu(D_{i})
     \end{align}
for any $(f,\mu)\in \bF\times \bM$ under the convention $D_0=\emptyset,$ 
where $\circ\colon [0,\infty)^2\to [0,\infty)$ and $\delta$ is a~dissimilarity function.  We split the study of connection between $\fC^{\circ,\delta}_{\mu}$ and \eqref{n:d1b} into two cases: $|X|=2$ and $|X|\ge 3.$

\begin{tw}
Let $X=[2].$ Assume that $\delta$ is a~dissimilarity function  and  $\circ\colon [0,\infty)^2\to [0,\infty)$ is nondecreasing.
\begin{itemize}[noitemsep]
    \item[(i)] The operator  defined in \eqref{n:d1b} coincides with         $\fC_{\mu}^{\circ,\delta}(f)$ for any $f\in\bF$ and any symmetric $\mu\in \bM$ if and only if  
    \begin{align}\label{n:3h}
     \delta(x_2,0)\circ y\le \delta(x_1,0)\circ y+\delta (x_2,x_1)\circ y
    \end{align}
    for any  $0\le x_1\le x_2<\infty$ and any $y\in [0,\infty).$
    \item[(ii)] If $0\circ a=0$ for all $a,$ then the operator  defined in \eqref{n:d1b} coincides with $\fC_{\mu}^{\circ,\delta}(f)$ for any $(f,\mu)\in\bF\times \bM$  if and only if  \eqref{n:3h} is true  for any  $0\le x_1\le x_2<\infty$ and any $y\in [0,\infty).$
\end{itemize}
\end{tw}
\begin{proof} 
(i) Since $X=[2],$ so $\cH_{\text{chain}}$ is given in Example~\ref{ex:3.1}\,(c).
Let $x_{(1)}=\min\{x_1,x_2\}$ and $x_{(2)}=\max \{x_1,x_2\}.$ Then the operators take the form
\begin{align}
\rCS^{\rL}_{\cH_{\text{chain}},\bA^{\inf},\bA^{\inf}}(f,\mu)&=\max \big\{\delta(x_{(2)},0)\circ \mu(\{(2)\}),\,\delta(x_{(1)},0)\circ \mu(X)+\delta(x_{(2)},x_{(1)})\circ \mu(\{(2)\})\big\},\label{n:4a}\\
\fC_{\mu}^{\circ,\delta} (f)&=\delta(x_{(1)},0)\circ \mu(X)+ \delta(x_{(2)},x_{(1)})\circ \mu(\{(2)\})\notag
\end{align}
for any symmetric $\mu\in\bM.$
Thus, the equality $\rCS^{\rL}_{\cH_{\text{chain}},\bA^{\inf},\bA^{\inf}}(f,\mu)=\fC_{\mu}^{\circ,\delta}(f)$ holds for any $f\in\bF$ and any symmetric $\mu\in\bM$ if and only if
\begin{align*}
    \delta(x_{(2)},0)\circ \mu(\{(2)\})\le \delta(x_{(1)},0)\circ \mu(X)+\delta(x_{(2)},x_{(1)})\circ \mu(\{(2)\})
\end{align*}
for any $0\le x_1\le x_2<\infty$ and any symmetric $\mu.$ As $\mu(\{2\})\le \mu(X)$ and $\circ$ is  nondecreasing, the latter condition is equivalent to $\delta(x_{(2)},0)\circ y\le \delta(x_{(1)},0)\circ y+\delta(x_{(2)},x_{(1)})\circ y$ for any $0\le x_1\le x_2<\infty$  and any $0\le y<\infty,$ which finishes the proof of part (i). 

The proof of part (ii) is similar to the proof of (i), since \eqref{n:4a} holds for any $\mu\in\bM$ in view of $0\circ a=0$ for any $a.$
\end{proof}

\begin{tw}\label{tw:4.13}
Let $|X|\ge 3.$ Assume that $\delta$ is a~dissimilarity function  and  $\circ\colon [0,\infty)^2\to [0,\infty)$ is nondecreasing such that $0\circ 0=0.$
Then the operator  defined in \eqref{n:d1b} coincides with $\fC_{\mu}^{\circ,\delta}(f)$ for any $f\in\bF$ and any  symmetric $\mu\in \bM$ if and only if $\delta(x_3,x_1) \circ y\le \delta(x_2,x_1)\circ y+\delta(x_3,x_2)\circ y$ for any $0\le x_1\le x_2\le x_3<\infty $  and any $y\in [0,\infty).$ 
\end{tw}
\begin{proof}  
``$\Rightarrow$'' 
Assume that $f(1)=x_1,$ $f(2)=x_2$  and $f(j)=x_3$ for $j\in\{3,\ldots,n\}$ such that $0\le x_1\le x_2\le x_3,$ where $n=|X|.$
Clearly,    
$$
\fC_{\mu}^{\circ,\delta} (f)=\delta(x_1,0)\circ \mu(\bs{1})+ \delta(x_2,x_{1})\circ \mu(\bs{2}) +\delta(x_3,x_2)\circ \mu(\bs{3})+\sum_{i=4}^n 0\circ \mu(\bs{i}).
$$
The operator \eqref{n:d1b} coincides with $\fC_{\mu}^{\circ,\delta},$
so  for $\cD_2=\{\bs{3},\bs{1}\}\in\cH_{\text{chain}},$ we get
$$
\delta(x_1,0)\circ \mu(\bs{1})+\delta(x_3,x_1)\circ \mu(\bs{3})\le \fC_{\mu}^{\circ,\delta}(f)
$$
for any symmetric $\mu\in\bM.$ 
In consequence, for any $y\in [0,\infty)$ and the symmetric monotone measure given by 
$$
\mu(E)=\begin{cases}
y, & \textrm{if}\, |E|\in \{n,n-1,n-2\}, \\ 
0, & \textrm{otherwise},
\end{cases}
$$ 
we obtain $\delta(x_3,x_1)\circ y\le \delta(x_2,x_{1})\circ y +\delta(x_3,x_2)\circ y,$
as $0\circ 0=0,$ which finishes the proof of the subadditivity condition.

\noi  ``$\Leftarrow$'' 
Since $\aAi[\bs{i}]{f}{\inf}=\f{i}$ for any $i\in [n],$  we get 
$$\sup_{\cD\in \cH_{\text{chain}}}\sum_{i=1}^l \delta\big(\aAi[D_i]{f}{\inf},\aAi[D_{i-1}]{f}{\inf}\big)  \circ \mu(D_{i})\ge \fC_{\mu}^{\circ,\delta}(f),
$$ 
so it is enough to show that 
   \begin{align*}
        L:=\sum_{i=1}^l \delta\big(\aAi[D_i]{f}{\inf},\aAi[D_{i-1}]{f}{\inf}\big)  \circ \mu(D_{i})\le \fC_{\mu}^{\circ,\delta}(f)
   \end{align*}
for any $\cD_l\in \cH_{\text{chain}}$ such that $l\in [n].$  
Then $L=S_1,$ where
   \begin{align*}
       S_k=\sum_{i=k}^l\delta \big(\fs{D_i},\fs{D_{i-1}}\big) \circ \mu(D_i)
   \end{align*}
for $k\in [l].$   Let $k_1=\min\{k\in [n]\colon \fs{D_1}=f_{(k)}\}.$ Thus $D_1\subseteq \bs{k_1}$ and
    \begin{align}\label{n3b}
        L\le \delta (\f{k_1},0)\circ \mu(\bs{k_1})+\delta(\fs{D_2},\f{k_1}) \circ \mu(D_2)+S_3.
    \end{align}
Then there exists $k_2\in \{k_1+1,\ldots,n\}$ such that $|D_2|=|\bs{k_2}|$ and  $\fs{D_2}\le \fs{\bs{k_2}}=f_{(k_2)}$  (see the proof of Theorem \ref{tw:4.3}). In view of  $\f{k_1}=\fs{D_1}\le \fs{D_2}$ (as $D_2\subset D_1$) and Definition~\ref{def:dis}\,(a) and (c), from  \eqref{n3b} and symmetricity of $\mu$ we get
    \begin{align}\label{n3c}
        L&\le \delta (\f{k_1},0)\circ \mu(\bs{k_1})+\delta(\f{k_2},\f{k_1}) \circ \mu(\bs{k_2})+\delta(\fs{D_{3}},\f{k_2}) \circ \mu(D_3)+S_4.
         \end{align}   
Now we  only  focus on $\delta (\f{k_1},0)\circ \mu(\bs{k_1})$ from \eqref{n3c}. By the assumption on $\delta$  and by monotonicity of $\mu$ we obtain 
    \begin{align}\label{n3d}
        \delta (\f{k_1},0)\circ \mu(\bs{k_1})&\le
        \delta(\f{1},\f{0})\circ \mu(\bs{k_1}) + \delta(\f{k_1},\f{1})\circ \mu(\bs{k_1})\notag
    \\&\le \delta(\f{1},\f{0})\circ \mu(\bs{1}) + \delta(\f{k_1},\f{1})\circ \mu(\bs{k_1})\notag
    \\&\le\ldots \le \sum_{i=1}^{k_1} \delta(\f{i},\f{i-1})\circ \mu(\bs{i}).
    \end{align}
Now consider the term $\delta(\f{k_2},\f{k_1}) \circ \mu(\bs{k_2})$ from \eqref{n3c}. Using the similar argument as above we get
\begin{align}\label{n3e}
    \delta(\f{k_2},\f{k_1}) \circ \mu(\bs{k_2})&\le
    \delta(\f{k_1+1},\f{k_1})\circ \mu(\bs{k_2}) + \delta(\f{k_2},\f{k_1+1})\circ \mu(\bs{k_2})\notag
    \\&\le \delta(\f{k_1+1},\f{k_1})\circ \mu(\bs{k_1+1}) + \delta(\f{k_2},\f{k_1+1})\circ \mu(\bs{k_2})\notag
    \\&\le \ldots\le \sum_{i=k_1+1}^{k_2} \delta(\f{i},\f{i-1}) \circ \mu(\bs{i}).
\end{align}
So putting \eqref{n3d}--\eqref{n3e} in \eqref{n3c} we have
\begin{align*}
    L&\le \sum_{i=1}^{k_2} \delta(\f{i},\f{i-1})\circ \mu(\bs{i}) + \delta(\fs{D_{3}},\f{k_2}) \circ \mu(D_3)+S_4.
\end{align*}
Repeating the same arguments for $D_3,\ldots,D_{l}$ leads to the inequality
    \begin{align*}
         L\le \sum_{i=1}^{k_l} \delta(\f{i},\f{i-1})\circ \mu(\bs{i})\le \fC_{\mu}^{\circ,\delta}(f),
    \end{align*}
which finishes the proof. 
\end{proof}

\begin{remark}
Note that the assumption on $|X|\ge 3$ and $0\circ 0=0$ is used only in the implication ``$\Rightarrow$'' in the proof of Theorem~\ref{tw:4.13}.
\end{remark}

From Theorem~\ref{tw:4.11} the two results connecting our approach with the existing ones in the literature for the set $X$ such that $|X|\ge 3$ follow:
\begin{enumerate}[noitemsep]
     \item[(a)] if $\delta\colon [0,1]^2\to [0,1]$ is a~restricted dissimilarity function,  then the Choquet-Sugeno-like operator  defined in \eqref{n:d1b} with $\circ=\cdot$ coincides with the $d$-Choquet integral  for any $f\in\bF^1$ and any symmetric capacity $\mu$ if and only if    $\delta(x_3,x_1) \le \delta(x_2,x_1)+\delta(x_3,x_2)$  for any $0\le x_1\le x_2\le x_3\le 1.$ 
     
    \item[(b)] if $\delta(x,y)=|x-y|,$ an operation $\circ\colon [0,1]^2\to [0,1]$ is nondecreasing and  $0\circ 0=0,$  then the operator \eqref{n:d1b} coincides with $\fC_{\mu}^\circ(f)$  for any $f\in\bF^1$ and any symmetric capacity $\mu$ if and only if $x\mapsto x\circ b$ is subadditive for any $b\in [0,1].$ 
\end{enumerate}

A~representation of $\fC_{\mu}^{\circ,\delta}$ using $t$-level sets reads as follows. The  result is  an easy consequence of the fact that $\delta(x,x)=0$ for any $x.$
\begin{proposition}
Let $\delta$ be a~dissimilarity function and $\circ\colon [0,\infty)^2\to [0,\infty).$ If $0\circ a=0$ for all $a,$ then
\begin{align}\label{n3f}
    \fC^{\circ,\delta}_{\mu}(f) =\sum_{i=1}^n \delta(\f{i},\f{i-1}) \circ \mu(\{f\ge \f{i}\})
\end{align}
for any $(f,\mu)\in \bF\times \bM,$  where $(\cdot)\colon [n]\to [n]$ is a~permutation such that $\f{1}\le\ldots\le \f{n}$ with $\f{0}=0.$  
\end{proposition}


It is clear that \eqref{n3f} is true for any $\mu\in\bM$ and any binary operation $\circ$ whenever $f\in\bF$ is decreasing or increasing. 
Again, it is not possible to obtain the equality \eqref{n3f} for any $f\in\bF$ and any binary operation $\circ$ even for symmetric $\mu\in\bM.$

\begin{example} 
Let $X,$ $f$ and $\mu$ be the same as in Example~\ref{ex:4.6}. Then 
\begin{align*}
    \sum_{i=1}^2\delta(\f{i},\f{i-1})\circ \mu(\{f\ge \f{i}\})&= \delta(\f{1},0)\circ \mu(X) +0\circ \mu(X),\\
    \sum_{i=1}^2\delta(\f{i},\f{i-1})\circ \mu(\bs{i})&=\delta(\f{1},0)\circ \mu(X) +0\circ \mu(\{2\}).
\end{align*}
Clearly, for $0\circ \mu(X)\neq 0\circ \mu(\{2\})$ the equality in \eqref{n3f} does not hold.
\end{example}

\subsection{\textbf{Connection with $\fC_{\mu}^{(\rF_1,\rF_2)}$}}
Let $\rF_1,\rF_2\colon [0,\infty)^2\to [0,\infty)$  and $\rF_1\ge \rF_2.$ 
The operator $\fC_{\mu}^{(\rF_1,\rF_2)}$
is defined as
\begin{align}\label{n:cf1}
    \fC_{\mu}^{(\rF_1,\rF_2)}(f)&=\sum_{i=1}^{n} \big(\rF_1 (\f{i},\mu(\bs{i}))-\rF_2(\f{i-1},\mu(\bs{i}))\big),
\end{align}
where $(\cdot)\colon [n]\to [n]$ is a~permutation such that $\f{1}\le\ldots\le \f{n}$  and $\bs{i}= \{(i),\ldots,(n)\}$ for $i \in [n]$  with $\f{0}=0.$
The formula \eqref{n:cf1} can be rewritten as follows
\small{\begin{align}\label{n:cf2}
     \fC_{\mu}^{(\rF_1,\rF_2)}(f)=\sum_{i=1}^{n-1} \big(\rF_1 (\f{i},\mu(\bs{i}))-\rF_2 (\f{i},\mu(\bs{i+1}))\big)+\rF_1(\f{n},\mu(\bs{n}))-\rF_2(0,\mu(\bs{1})).
\end{align}}\normalsize
The  operator \eqref{n:cf1} is well defined (i.e., independent on permutations) when either  $\rF_1=\rF_2$ or $f\in\bF$ is monotone (decreasing or increasing) or $\mu\in\bM$  is symmetric.

Putting  $\rF{_1}|_{[0,1]^2}=\rF{_2}|_{[0,1]^2}=\rC$ in \eqref{n:cf1}, where $\rC$ is a~copula, we get the operator firstly defined in \cite{klement10} and then redefined in \cite{lucca} under the name $\rC\rC$-integral (or
\textit{Choquet-like Copula-based integral}) of $(f,\mu)\in \bF^1\times \bM^1$ in the form
\begin{align}\label{n:cf1a}
  \fC_{\mu}^{(\rC,\rC)}(f)=\sum_{i=1}^{n} \big(\rC(\f{i},\mu(\bs{i}))-\rC (\f{i-1},\mu(\bs{i}))\big),
\end{align}
where $\f{i}$ and $\bs{i}$ are the same as in~\eqref{n:cf1}.
The operator \eqref{n:cf1a} is a~$[0,1]$-valued universal integral (\cite{klement10}). Setting $\rF_1,\rF_2\colon [0,1]^2\to [0,1]$ in \eqref{n:cf1} such that $\rF_1\ge \rF_2$ and  $\rF_1(a,\mu(X))=a$ for all $a,$ we get 
\begin{align*}
   \fC_{\mu}^{(\rF_1,\rF_2)}(f)=\f{1}+\sum_{i=2}^{n} \big(\rF_1 (\f{i},\mu(\bs{i}))-\rF_2 (\f{i-1},\mu(\bs{i}))\big)
\end{align*}
for any $(f,\mu)\in \bF^1\times \bM^1,$ being a~version of $C_{\rF_1\rF_2}$-integral defined in \cite[Definition 7]{lucca2019}.

Let $\cH=\cH_{\text{chain}}$ and $\cR=\cR^-.$
Then the Choquet-Sugeno-like operator \eqref{n:cs1} with $\rL(x,y,z,w)=\rF_1(x,z)-\rF_2(y,z)$ and $\bA=\bA^{\inf}=\widehat{\bA}$ can be written as follows
\begin{align}\label{n:cf3}
    \rCS^{(\rF_1,\rF_2)}_{\cH_{\text{chain}},\bA^{\inf},\bA^{\inf}}(f,\mu)=\sup_{\cD\in \cH_{\text{chain}}}\sum _{i=1}^l
    \big(\rF_1(\aAi[D_i]{f}{\inf},\mu(D_i))-\rF_2(\aAi[D_{i-1}]{f}{\inf},\mu(D_i))\big)
\end{align}
under the convention $D_0=\emptyset.$

A~pair $(\rF_1,\rF_2)$ of functions $\rF_1,\rF_2\colon E_1\times E_2\to [0,\infty),$ where $E_1, E_2\subseteq [0,\infty),$ is \textit{pairwise $2$-increasing} if
\begin{align*}
       \rF_1(x_1,y_2)-\rF_2(x_1,y_1)\le \rF_1(x_2,y_2)-\rF_2(x_2,y_1)
\end{align*}
for any $[x_1,x_2]\times [y_1,y_2] \subseteq E_1\times E_2.$ 
It is clear that the pair $(\rC,\rC)$ is pairwise $2$-increasing whenever $\rC$ is a~copula \cite{nelsen}. 


\begin{tw}\label{tw:4.11}
Assume that $\rF_1,\rF_2\colon [0,\infty)^2\to [0,\infty)$ are pairwise $2$-increasing,  $\rF_1$ is nondecreasing, $\rF_1\ge \rF_2$ and $\rF_2(0,b)=0$ for any $b.$
Then the operator \eqref{n:cf3} coincides with $\fC_{\mu}^{(\rF_1,\rF_2)}(f)$ for any  $f\in\bF$ and any symmetric $\mu\in\bM.$ 
\end{tw}

\begin{proof}
By \eqref{n:cf1} and \eqref{n:cf3} we get $\rCS^{(\rF_1,\rF_2)}_{\cH_{\text{chain}},\bA^{\inf},\bA^{\inf}}(f,\mu)\ge \fC_{\mu}^{(\rF_1,\rF_2)}(f),$ so we will prove the reverse inequality, that is, 
   \begin{align*}
   L:=\sum _{i=1}^l
   \big(\rF_1(\aAi[D_i]{f}{\inf},\mu(D_i))-\rF_2(\aAi[D_{i-1}]{f}{\inf},\mu(D_i))\big)\le \fC_{\mu}^{(\rF_1,\rF_2)}(f)
   \end{align*}
for any $\cD_l\in \cH_{\text{chain}}$ with $l\in [n].$ Since
$$\sum_{i=1}^l\bigl(\rF_1(a_i,b_i)-\rF_2(a_{i-1},b_i)\bigr) = \sum_{i=1}^{l-1} \bigl(\rF_1(a_i,b_i)-\rF_2(a_i,b_{i+1})\bigr)+\rF_1(a_l,b_l)- \rF_2(a_0,b_1)$$
for any nonnegative sequence $(a_i)_{i=0}^l$ and $(b_i)_{i=1}^l,$
we get $L=S_1,$  where
\begin{align*}
    S_k=\sum_{i=k}^{l-1}  \big(\rF_1(\aAi[D_i]{f}{\inf},\mu(D_i))-\rF_2(\aAi[D_{i}]{f}{\inf},\mu(D_{i+1}))\big)+\rF_1(\aAi[D_l]{f}{\inf},\mu(D_l))
\end{align*}
for $k\in [l],$ as $\rF_2(0,\mu(D_1))=0.$ 
Let $k_1=\min\{k\colon  \fs{D_1}=f_{(k)}\}.$ Clearly, $D_1\subseteq \bs{k_1}.$ Hence by monotonicity of $\rF_1$
we have
\begin{align}\label{n:cf5}
L&\le \rF_1(\f{k_1},\mu(\bs{k_1}))-\rF_2(\f{k_1},\mu(D_{2}))+S_2.
\end{align}
Then there exists $k_2\in \{k_1+1,\ldots,n\}$ such that $|D_2|=|\bs{k_2}|$ and $\fs{D_2}\le\fs{\bs{k_2}}$ (see the proof of Theorem \ref{tw:4.3}). 
Thus 
    \begin{align*}
       L&\le \rF_1(\f{k_1},\mu(\bs{k_1}))-\rF_2(\f{k_1},\mu(\bs{k_2}))
       +\rF_1(\aAi[D_2]{f}{\inf},\mu(D_2))-\rF_2(\aAi[D_{2}]{f}{\inf},\mu(D_{3}))+S_3.
   \end{align*}
    Since $\rF_1$ and $\rF_2$ are  pairwise $2$-increasing, $\mu(D_3)\le \mu(D_2)$ and $\fs{\bs{k_2}}=f_{(k_2)},$ we get
    \begin{align}\label{n:cf4a}
        L&\le \rF_1(\f{k_1},\mu(\bs{k_1}))-\rF_2(\f{k_1},\mu(\bs{k_2}))\notag
       +\rF_1(f_{(k_2)},\mu(D_2))-\rF_2(f_{(k_2)},\mu(D_{3}))+S_3\notag
       \\&=\rF_1(\f{k_1},\mu(\bs{k_1}))-\rF_2(\f{k_1},\mu(\bs{k_2}))
       +\rF_1(f_{(k_2)},\mu(\bs{k_2}))-\rF_2(f_{(k_2)},\mu(D_{3}))+S_3.
       \end{align}
   Let $M=\rF_1(\f{k_1},\mu(\bs{k_1}))-\rF_2(\f{k_1},\mu(\bs{k_2})).$ Then by $\rF_1\ge \rF_2$ and pairwise $2$-increasingness we get
   \begin{align}\label{n:cf4b}
       M&\le \rF_1(\f{k_1},\mu(\bs{k_1}))-\rF_2(\f{k_1},\mu(\bs{k_1+1}))+\rF_1(\f{k_1},\mu(\bs{k_1+1}))-\rF_2(\f{k_1},\mu(\bs{k_2}))\notag
       \\&\le \rF_1(\f{k_1},\mu(\bs{k_1}))-\rF_2(\f{k_1},\mu(\bs{k_1+1}))+\rF_1(\f{k_1+1},\mu(\bs{k_1+1}))-\rF_2(\f{k_1+1},\mu(\bs{k_2}))\notag
       \\&\le \ldots\le \sum_{i=k_1}^{k_2-1} \big(\rF_1(\f{i},\mu(\bs{i}))-\rF_2(\f{i},\mu(\bs{i+1}))\big).
   \end{align}
Combining \eqref{n:cf4a} and \eqref{n:cf4b} we get
      \begin{align*}
          L&\le \sum_{i=k_1}^{k_2-1} \big(\rF_1(\f{i},\mu(\bs{i}))-\rF_2(\f{i},\mu(\bs{i+1}))\big)
         +\rF_1(f_{(k_2)},\mu(\bs{k_2}))-\rF_2(f_{(k_2)},\mu(D_{3}))+S_3.
      \end{align*}
By repeating the same arguments, we obtain  
\begin{align*}
    L&\le \sum_{i=k_1}^{k_l-1}\big( \rF_1(\f{i},\mu(\bs{i}))-\rF_2(\f{i},\mu(\bs{i+1}))\big)+\rF_1(\f{k_l},\mu(\bs{k_l})) \le \fC_{\mu}^{(\rF_1,\rF_2)}(f),
\end{align*}
where in the last inequality the formula \eqref{n:cf2} has been used. This completes the proof.  
\end{proof}

Let $\rC$ be a~copula.
Similarly as in the proof of Theorem~\ref{tw:4.11}, one can prove that the operator defined in \eqref{n:cf3} with  $\rF_{1}|_{[0,1]^2}=\rF_2|_{[0,1]^2}=\rC$ coincides with $\fC_{\mu}^{(\rC,\rC)}(f)$  ($\rC\rC$-integral) for any  $f\in\bF^1$ and any symmetric capacity $\mu.$

The operator $\fC_{\mu}^{(\rF_1,\rF_2)}$ can also be represented via $t$-level sets. Indeed, the equality
\begin{align}\label{n:4c}
   \fC_{\mu}^{(\rF_1,\rF_2)}(f)=\sum_{i=1}^{n}\big(\rF_1(\f{i},\mu(\{f\ge \f{i}\}))-\rF_2(\f{i-1},\mu(\{f\ge \f{i}\}))\big),
\end{align} 
where $(\cdot)\colon [n]\to [n]$ is a~permutation such that $0=f_{(0)}\le f_{(1)}\le \ldots \le f_{(n)},$  holds 
\begin{itemize}[noitemsep]
    \item for any $(f,\mu)\in \bF\times \bM$ whenever $\rF_1=\rF_2$, or
    \item for any $\mu\in\bM$ and any $\rF_1,\rF_2$ such that $\rF_1\ge \rF_2$ whenever $f\in\bF$ is decreasing or increasing.
\end{itemize}

Observe that the equality \eqref{n:4c} need not hold for any $f\in\bF$ and any $\rF_1,\rF_2$ such that $\rF_1\ge \rF_2$ even for symmetric $\mu\in\bM.$

\begin{example} Let $X,$ $f$ and $\mu$ be the same as in Example~\ref{ex:4.6}. Then
\begin{align*}
    \sum_{i=1}^2\big(\rF_1(\f{i},\mu(\{f\ge \f{i}\}))-\rF_2(\f{i-1},\mu(\{f\ge \f{i}\}))\big)&= \rF_1(0.5,2)-\rF_2(0,2)+\rF_1(0.5,2)-\rF_2(0.5,2),\\
   \sum_{i=1}^{2} \big(\rF_1 (\f{i},\mu(\bs{i}))-\rF_2(\f{i-1},\mu(\bs{i}))\big)&=\rF_1(0.5,2)-\rF_2(0,2)+\rF_1(0.5,1)-\rF_2(0.5,1).
\end{align*}
The equality in \eqref{n:4c} does not hold 
e.g., if $\rF_1=\rF$ and $\rF_2=c\rF$ for any $c\in (0,1)$ and any $\rF\colon [0,\infty)^2\to [0,\infty)$ such that $\rF(0.5,2)\neq \rF(0.5,1)$.
\end{example}


\subsection{\textbf{Connection with $\mathfrak{IE}_\mu^{\circ,\sI}$}}
Let $\circ\colon [0,\infty)\times (-\infty,\infty)\to (-\infty,\infty).$
Define the following operator 
\begin{align}\label{ie:n1}
   \mathfrak{IE}_\mu^{\circ,\sI}(f) = \sum_{\emptyset\neq D\subseteq X} \sI(f,D)\circ \mob (D)
\end{align}
for any $(f,\mu)\in \bF\times \bM,$ 
where $\sI\colon \bF\times \Sigma_0 \to [0,\infty)$ is an \textit{extended interaction operator} satisfying the conditions: 
\begin{itemize}[noitemsep]
    \item[$(I1)$] $\sI(f,\{i\}) = f(i)$ for any $f\in\bF$ and any $i\in [n]$; 
    \item[$(I2)$] $\sI(f,D) \le \sI(g,D)$ for any $D\in\Sigma_0$ and any $f, g\in\bF$ such that $f(x)\le g(x)$ for all $x\in D$;
    \item[$(I3)$] $\sI(f,D) \ge \sI(f,E)$ for any $f\in\bF$ and any $D,E\in\Sigma_0$ with $D\subseteq E.$
\end{itemize} 
Putting $\circ=\cdot$ and $\sI$ being an interaction operator (\cite[Definition 5]{Honda2017}) in \eqref{ie:n1}, we get a~\textit{discrete inclusion-exclusion integral} (\cite[Theorem 3]{Honda2017}).
For a~connection between $\mathfrak{IE}^{\circ,\sI}_\mu$ and the Choquet-Sugeno-like operator of the form \eqref{n:chm1} we first explain a~relationship between the extended interaction operator and conditional aggregation operator. 

\begin{example}
\begin{itemize}[noitemsep]
    \item[(i)] The mapping $\aA[D]{f}=(\textstyle{\sum_{i\in D}f(i)^p})^{1/p}$ with $p\in [1,\infty)$ is the conditional aggregation operator which is not interaction operator since it violates $(I3).$ 
    \item[(ii)] The \L ukasiewicz t-conorm $\aA[D]{f}=1\wedge \textstyle{\sum_{i\in D}f(i)}$ is the conditional aggregation operator violating the properties $(I1)$ and $(I3).$ 
\end{itemize}
\end{example}
These examples with the next proposition claim that the concept of conditional aggregation is more general than that of extended interaction operators. For that purpose we introduce the following notion: a~conditional aggregation operator $\aA[\cdot]{\cdot}$ is called \textit{conjunctive} if $\aA{f}\le \aAi{f}{\inf}$ for any $D\in\Sigma_0$ and any $f\in\bF.$ 

\begin{proposition}\label{pro:4.12}
Let $D\in \Sigma_0.$ Then $\sI(f,D)=\aA[D]{f}$ for any $f\in \bF,$ where $\sA$ is a~conjunctive conditional aggregation operator.
\end{proposition}

\proof In order to prove that an~extended interaction operator $\sI$ is a~conditional aggregation operator, only condition $(C2)$ has to be verified. Let $D\in \Sigma_0.$ Clearly, $D$ can be written as a~finite union of singletons, let us say $D=\textstyle{\bigcup_{j\in D}\{j\}}.$
Then for any $f\in\bF$ from $(I3)$ and $(I1)$ we have $\sI(\mI{D^c},D)\le \sI(\mI{D^c},\{j\})=\mI{D^c}(j)=0$ for each $j\in D.$
This yields that $\sI$ satisfies $(C2).$ To prove conjunctivity, for $D\in\Sigma_0$ put $f(j) = \textstyle{\min_{i\in D} f(i)}.$ Then $j\in D,$ and
$$
\sI(f,D)\le\sI(f,\{j\}) =  f(j) = \aAi[D]{f}{\inf}.
$$ 
This completes the proof.
\qed


Using Proposition~\ref{pro:4.12}, the connection between $\mathfrak{IE}^{\circ,\sI}_\mu$ and \eqref{n:chm1} is immediate.

\begin{tw}
Let $\circ\colon [0,\infty)\times (-\infty,\infty)\to (-\infty,\infty)$ and $\bA$ be a~FCA consisting of all conjunctive conditional aggregation operators. Then the operator $\mathfrak{IE}^{\circ,\sI}_{\mu}(f)$ with $\mathsf{I}\in\bA$ coincides with the Choquet-Sugeno-like operator \eqref {n:chm1} for any $(f,\mu)\in\bF\times \bM.$
\end{tw}

\section{Selected properties of the Choquet-Sugeno-like operator}\label{sec:property}

In this section several properties of the Choquet-Sugeno-like operator will be examined with respect to function $\rL.$ Let us summarize the functions $\rL=\rL(x,y,z,w)$ that have been used in the paper:
\begin{alignat*}{3}
 1.\quad & \rL_1(x,y,z,w)= |x-y|^p \quad && \text{(Example~\ref{ex:3.8});}\\
 2. \quad & \rL_2(x,y,z,w)= x\circ z \quad && \text{(Examples ~\ref{ex:3.3}--\ref{ex:3.5} with  $\rF=\circ$);}\\
 3. \quad & \rL_3(x,y,z,w)= y\circ w \quad && \text{(formula \eqref{n:chm1});}\\
 4. \quad & \rL_4(x,y,z,w)=x\circ (z-w)_+  \quad && \text{(formulas \eqref{n:c1b} and \eqref{n:c5a})};\\ 
 5. \quad & \rL_5(x,y,z,w)= (x-y)_+\circ z  \quad && \text{(formula \eqref{n:r1});}\\
 6. \quad & \rL_6(x,y,z,w)=\delta(x,y)\circ z  \quad && \text{(formula \eqref{n:d1b})};\\
 7. \quad & \rL_7(x,y,z,w)=\rF_1(x,z)-\rF_2(y,z)  \quad &&  \text{(formula \eqref{n:cf3})},
 \end{alignat*}
 where $\delta,\circ,\rF_1,\rF_2$ have been defined at the indicated places of the article.
To simplify the notation, we will write just $\rCS^{\rL}(f)$ instead of $\rCS^{\rL}_{\cH,\bA,\widehat{\bA}}(f,\mu,\wmu).$
Proofs of the results presented in this section are immediate when representing the Choquet-Sugeno-like operator \eqref{n:cs1} in the following way
\begin{align}\label{pro:n1}
    \rCS^{\rL}(f)&=\sup_{\cD\in\cH} \Big\{\sum_{(C,D)\in\cR,\,C,D\neq\emptyset}\rL\big(\aA[C]{f},\waA[D]{f},\mu(C),\wmu(D)\big)+\sum_{(C,\emptyset)\in\cR,\,C\neq\emptyset}\rL\big(\aA[C]{f},0,\mu(C),0\big)\notag
    \\&\qquad\qquad+\sum_{(\emptyset,D)\in\cR,\,D\neq \emptyset}\rL\big(0,\waA[D]{f},0,\wmu(D)\big)+\sum_{(\emptyset,\emptyset)\in\cR}\rL(0,0,0,0)\Big\},
\end{align}
as $\aA[\emptyset]{\cdot}=\waA[\emptyset]{\cdot}=0$ and $\wmu(\emptyset)=0.$

\begin{proposition}\label{pro:5.1}
If $\rL(0,0,z,w)=0$ for all $z,w,$  then $\rCS^{\rL}(0_X)=0.$ 
\end{proposition}
\begin{proof}
The statement follows from \eqref{pro:n1} and the fact that $\aA{0_X}=0=\widehat{\sA}(0_X|D)$ for any $D\in \Sigma_0,$ where $\sA \in \bA$ and $\widehat{\sA}\in \widehat{\bA}$ (see \cite[Proposition 3.3\,(b)]{BHHK21}).
\end{proof}

\begin{example}
The assumption in Proposition~\ref{pro:5.1} is satisfied for the following functions:
\begin{itemize}[noitemsep]
   \item $\rL_1$;
    \item $\rL_2-\rL_6$ whenever $0\circ b=0$ for any $b$;
    \item $\rL_7$ whenever $\rF_1(0,b)=\rF_2(0,b).$
  \end{itemize}
\end{example}

\begin{proposition}\label{pro:5.3}(Monotonicity) Assume that  $\rL(x,y,z,w)$ is nondecreasing in $x$ and $y$ for any fixed $z,w.$ Then
$ \rCS^{\rL}(f)\le \rCS^{\rL}(g)$ whenever $f\le g.$
\end{proposition}
\begin{proof}
Use the equality \eqref{pro:n1} and the property $(C1)$ (see Definition \ref{def:conagr}).
\end{proof}

\begin{example}
The assumption of Proposition \ref{pro:5.3} is satisfied for functions $\rL_2-\rL_4$ whenever $x\mapsto x\circ b$ is nondecreasing for any $b.$ 
\end{example}

We say that a~conditional aggregation operator $\aA[D]{\cdot}$ is \textit{homogeneous} if $\aA[D]{\alpha f}=\alpha \aA[D]{f}$ for any $\alpha\in [0,\infty)$ and $f\in\bF.$ A~family $\bA$ of conditional aggregation operators is said to be \textit{homogeneous} if  $\aA[D]{\cdot}\in\bA$ is homogeneous for any  $D\in\Sigma_0.$

\begin{proposition}\label{pro:5.7}(Homogeneity)
Let $\bA$ and $\widehat{\bA}$ be homogeneous FCAs. If 
$\rL(\alpha x,\alpha y,z,w)= \alpha\rL(x,y,z,w)$ for all $\alpha\in [0,\infty)$ and all $x,y,z,w,$  then 
$f\mapsto \rCS^{\rL}(f)$ is a~homogeneous operator.
\end{proposition}

\begin{example}
The assumption is satisfied for:
\begin{itemize}[noitemsep]
    \item $\rL_1$ for $p=1$; 
    \item $\rL_2-\rL_5$ if $(\alpha a)\circ b=\alpha (a\circ b)$ for any $\alpha,a,b$;   
    \item  $\rL_6$ if $\delta(\alpha x, \alpha y)\circ z=\alpha (\delta(x,y)\circ z)$ for any $\alpha,x,y,z$;
    \item $\rL_7$ if $x\mapsto \rF_1(x,b)$ and $x\mapsto\rF_2(x,b)$ are homogeneous for any $b.$
\end{itemize}
\end{example}

We say that a~conditional aggregation operator $\aA[D]{\cdot}$ is \textit{subadditive} if $\aA[D]{f+g}\le  \aA[D]{f}+\aA[D]{g}$ for any $f,g\in\bF.$ A~FCA $\bA$ is said to be \textit{subadditive}  if  $\aA[D]{\cdot}\in\bA$ is subadditive  for any $D\in\Sigma_0.$

\begin{proposition}(Subadditivity)\label{P1}
Suppose that $\bA$ and $\widehat{\bA}$ are subadditive FCAs. If  $\rL(x+a,y+b,z,w)\le \rL(x,y,z,w)+\rL(a,b,z,w)$ for all $x,y,z,w,a,b,$ then 
the mapping $f\mapsto \rCS^{\rL}(f)$ 
is subadditive.
\end{proposition}

\begin{example}
The following functions $\rL$ are subadditive for the first and second coordinate:
\begin{itemize}[noitemsep]
    \item $\rL_1$ with $p=1$;
    \item $\rL_2-\rL_4$ are subadditive if $x\mapsto x\circ b$ is subadditive for any $b$;
    \item $\rL_5$ if $x\mapsto x\circ z$ is nondecreasing and subadditive for any $z$;
    \item $\rL_6$ if $\delta(x+a,y+b)\circ z\le \delta(x,y)\circ z +\delta(a,b)\circ z$ for any $x,y,z,a,b.$
\end{itemize}
\end{example}

\begin{proposition}\label{convexity}(Convexity) 
Let $f\mapsto\aA{f}$ and $f\mapsto \waA{f}$ be convex for any $D\in\Sigma_0,$ where $\sA\in \bA$ and $\widehat{\sA}\in \widehat{\bA}.$ If 
$\rL(\lambda x+(1-\lambda)a,\lambda y+(1-\lambda)b,z,w)\le  \lambda\rL(x,y,z,w)+(1-\lambda)\rL(a,b,z,w)$ for all $\lambda\in (0,1)$ and all $x,y,z,w,a,b,$ then $f\mapsto \rCS^{\rL}(f)$ is convex.
\end{proposition}

\begin{example}
The assumption in Proposition~\ref{convexity} is satisfied for
\begin{itemize}[noitemsep]
    \item $\rL_1$ if $p\ge 1$;
    
    \item $\rL_2-\rL_4$ whenever $x\mapsto x\circ b$ is convex for any $b$;
    
    \item $\rL_5$  whenever $x\mapsto x\circ z$ is nondecreasing and convex for any $z.$
\end{itemize}
\end{example}

We say that a~FCA $\bA$ is \textit{idempotent} if $\aA{b\mI{X}}=b$ for any $b\in [0,\infty)$ and any $\aA[D]{\cdot}\in \bA$ with $D\in\Sigma_0$ (see \cite[Proposition 3.10]{BHHK21}).
Obviously, $\bA^{\inf}$ and $\bA^{\sup}$ are idempotent FCAs.

\begin{proposition}(Idempotency)
Let $\bA$ and 
$\widehat{\bA}$ be idempotent FCAs.
Then $\rCS^{\rL}(b\mI{X})=b$ 
for all $b\in [0,\infty)$ if and only if 
\begin{align}\label{p:1}
    \sup_{\cD\in \cH} \Big\{&\sum_{(C,D)\in \cR,\,D\neq \emptyset} \rL(b,b,\mu(C),\wmu(D))+\sum_{(C,\emptyset)\in\cR,\,C\neq \emptyset}\rL(b,0,\mu(C),0)\notag
    \\& +\sum_{(\emptyset,D)\in\cR,\,D\neq \emptyset}\rL(0,b,0,\wmu(D))+\sum_{(\emptyset,\emptyset)\in\cR}\rL(0,0,0,0)\Big\}=b
\end{align}
  for any $b\in [0,\infty).$
\end{proposition}

For the special choice of L's the condition \eqref{p:1} can be simplified as the following examples demonstrate.
\begin{itemize}[noitemsep]
    \item Consider $\rL_1.$ If $\bA$ and $\widehat{\bA}$ are idempotent FCAs, then the $p$-variation of any constant function is equal to 0 (see Example~\ref{ex:3.8}), so the $p$-variation  is not idempotent, as expected.

    \item Consider $\rL_2$ such that $x\circ 0=0=0\circ z$ for any $x,z.$ Then \eqref{p:1} takes the form
    \begin{align}\label{p:2}
        \sup_{\cD\in \cH} \sum_{(C,D)\in \cR,\,C\neq \emptyset} b\circ \mu(C)=b.
    \end{align}
    If $\cR=\{(D,D)\colon D\in \cD\}$ for  $\cD\in\cH=\{\{D\}\colon D\in \Sigma_0\}$ (see Example~\ref{ex:3.3}), then \eqref{p:2} can be rewritten as follows $\textstyle{\sup_{C\in \Sigma_0} \{b\circ \mu(C)\}}=b$ for any $b.$ For $\rL_3$ we can apply a~similar approach.

    \item Consider 
    $\rL_4,$ $\cR=\cR^+$ and $\cH=\cH_{\text{chain}}$ (cf. section devoted to $\fC^{\mu,\wmu}_{\circ}$).  Then the condition \eqref{p:1} has the following form    
    \begin{align}\label{p:3}
        \sup_{\cD\in\cH_{\text{chain}}} \sum_{i=1}^l b\circ (\mu(D_i)-\wmu(D_{i+1}))=b
    \end{align}
    for any $b.$ The condition \eqref{p:3} is satisfied if $\circ=\cdot$ and $\mu=\wmu\in\bM^1$ (cf. \cite[Proposition 14]{horanska2018}).
    
    \item Consider $\rL_5,$ $0\circ z=0$ for any $z$ and $\cR=\{(D,D^c)\colon D\in \cD\}$ for $\cD\in \cH_{\text{one}}$ (cf. Example~\ref{ex:3.7}). Then \eqref{p:1} reduces to the equality
    $b\circ \mu(X)=b$  for any $b\in [0,\infty).$ 
    
    \item Consider $\rL_6,$ $0\circ z=0$ for any $z,$ $\cH=\cH_{\text{chain}}$ and $\cR=\cR^-$ (cf. section devoted to $\fC^{\circ,\delta}_{\mu}$). Then the condition \eqref{p:1} takes the form
    \begin{align}\label{id:n2}
        \sup_{D\in \Sigma_0}\{\delta(b,0)\circ \mu(D)\}=b
    \end{align}
    for any $b.$ The condition \eqref{id:n2} is true for nondecreasing $x\mapsto a\circ x$ for any $a$ such that $\delta(a,0)\circ \mu(X)=a$ for any $a$ (cf. \cite[Theorem 3.31]{bustince2020}).
    
    \item Consider $\rL_7,$ $\cH=\cH_{\text{chain}}$ and $\cR=\cR^-$ (cf. section devoted to $\fC^{(\rF_1,\rF_2)}_{\mu}$). Then the formula \eqref{p:1} takes the form
    \begin{align*}
        \sup_{\cD\in \cH_{\text{chain}}} \sum_{i=1}^l \bigl(\rF_1(b,\mu(D_i))-\rF_2(b,\mu(D_i))\bigr)=b
    \end{align*}
    for any $b\in [0,\infty).$
\end{itemize}

\section{Conclusion}
In this paper we have indicated a~way how to look at different operators with respect to a~nonadditive measure from a~new (in some sense unified) perspective. We have defined an operator generalizing Sugeno-like operator (Example~\ref{ex:3.3}), generalized Lebesgue integral (Example~\ref{ex:3.4}), $\rF$-decomposition integral with respect to a~finite partition decomposition system (Example~\ref{ex:3.5}), Lov\'{a}sz extension (Example~\ref{ex:3.6}) and generalized $p$-variation (Example~\ref{ex:3.8}). Moreover, we have given relationships with several functionals generalizing the discrete Choquet integral expressions
\begin{alignat*}{3}
(C)\intl_X f\md\mu&
=\sum_{i=1}^n \f{i} (\mu(\bs{i})-\mu(\bs{i+1}))\qquad&\text{(cf. $\fC_{\circ}^{\mu,\wmu}(f)$)},
\\&=\sum_{i=1}^n (\f{i}-\f{i-1})\mu(\bs{i})\quad &\text{(cf. $\fC_{\mu}^{\circ,\delta}(f)$)},
\\&=\sum_{i=1}^n \bigl(\f{i} \mu(\bs{i})-\f{i-1}\mu(\bs{i})\bigr)\qquad\qquad &\text{(cf. $\fC_{\mu}^{(\rF_1,\rF_2)}(f)$)}. 
\end{alignat*}
All of this has been possible thanks to the conditional aggregation operator defined in \cite{BHHK21} and the dependence relation between conditional sets proposed in the present paper.
The relation can generate different aggregation styles in decision making  and can be used in as diverse areas as graph theory, neural networks and fuzzy theory. Thus, describing new relations between conditional sets may produce new operators interesting both from a~theoretical point of view as well as for applications. 

\section*{Declaration of Competing Interest}

The authors declare that they have no known competing financial interests or personal
relationships that could have appeared to influence the work reported in this paper.


\end{document}